\documentclass[10pt]{amsart}
\usepackage{cases}
\usepackage[utf8]{inputenc}

\usepackage{amsmath,amsthm,amssymb,amscd}
\usepackage{stmaryrd}
\usepackage[all,cmtip]{xy}
\usepackage{mathrsfs}
\usepackage{color}
\usepackage{amsxtra}
\usepackage{amsfonts}
\usepackage{tikz-cd}
\usepackage{bbm}
\usepackage{hyperref}
\usepackage[mathscr]{eucal}
\usepackage[english]{babel}
\usepackage[autostyle]{csquotes}
\usepackage[all]{xy}
\usepackage{mathtools}
\usepackage{wasysym}

\numberwithin{equation}{section}
\newtheorem{theorem}[equation]{Theorem}
\newtheorem{thm}{Theorem}
\theoremstyle{plain}
\newtheorem{lemma}[equation]{Lemma}
\newtheorem{proposition}[equation]{Proposition}
\newtheorem{definition}[equation]{Definition}
\newtheorem{corollary}[equation]{Corollary}

\newtheorem*{corollary*}{Corollary}

\newtheorem{remark}[equation]{Remark}

\newenvironment{myproof}[2] {\emph{Proof of {#1} {#2}.}}{\hfill$\square$}

\def\GL{\mathrm{GL}}
\def\GSp{\mathrm{GSp}}
\def\GSpin{\mathrm{GSpin}}
\def\Sp{\mathrm{Sp}}
\def\GU{\mathrm{GU}}

\def\Nilp{(\mathrm{Nilp})}
\def\det{\mathrm{det}}
\def\Lie{\mathrm{Lie}}
\def\inv{\mathrm{inv}}

\def\Ta{\mathrm{Ta}}
\def\loc{\mathrm{loc}}

\DeclareMathOperator{\End}{End}
\DeclareMathOperator{\Adm}{Adm}

\def\calF{\mathcal{F}}

\def\calN{\mathcal{N}}
\def\calM{\mathcal{M}}
\def\calO{\mathcal{O}}

\def\gothS{\mathfrak{S}}

\def\AAA{\mathbb{A}}

\def\CC{\mathbb{C}}
\def\DD{\mathbb{D}}
\def\FF{\mathbb{F}}
\def\GG{\mathbb{G}}

\def\PP{\mathbb{P}}
\def\QQ{\mathbb{Q}}
\def\RR{\mathbb{R}}

\def\XX{\mathbb{X}}
\def\ZZ{\mathbb{Z}}

\newcommand{\Spf}[1]{\mathrm{Spf} (#1)}

\newcommand{\Sh}{\gothS h}

\usepackage[top=1.5in, bottom=1.5in, left=1.4in, right=1.4in]{geometry}
\usepackage{fancyhdr}
\address{\parbox{\linewidth} {Haining Wang,\\ Department of Mathematics,\\ McGill University,\\ 805 Sherbrooke St W,\\ Montreal, QC H3A 0B9, Canada.~ }}
\email{wanghaining1121@outlook.com}

\date{\today}

\begin{document}

\title[Quaternionic unitary Rapoport-Zink space]{On the Bruhat-Tits stratification of a quaternionic unitary Rapoport-Zink space}

\author{Haining Wang}
\begin{abstract}
In this article we study the special fiber of the Rapoport-Zink space attached to a quaternionic unitary group. The special fiber is described using the so called Bruhat-Tits stratification and is intimately related to the Bruhat-Tits building of a split symplectic group. As an application we describe the supersingular locus of the related Shimura variety.
\end{abstract}
\maketitle 
\tableofcontents
\section{Introduction}
\subsection{Motivations} In this article we study the basic locus of a specific Shimura variety. Namely the Shimura variety associated to a quaternionic unitary group. This Shimura variety is particular interesting as it is of PEL-type and is intimately related to Siegel and orthogonal type Shimura varieties. This article is inspired by the paper \cite{KR-ASENS94} where the basic locus of this Shimura variety is studied at a prime $p$ where the quaternion algebra is split. Using this description, special cycles on this Shimura variety are defined and their intersection numbers are related to Fourier coefficients of Eisenstein series. In this case the Shimura variety has good reduction and the quaternionic unitary group at $p$ agrees with the symplectic group of degree $4$. Therefore the supersingular locus can be studied essentially in the same way as in the case of a Siegel threefold of hyperspecial level. An explicit description of this supersingular locus is available in \cite{KO-COM87}. However in \cite{KR-ASENS94} the authors are able to relate the description to the Bruhat-Tits building of an inner form of the symplectic group. In this article we treat the case where the quaternion algebra is ramified at $p$ and we describe the supersingular locus with its relation to Bruhat-Tits building. Note in this case the Shimura variety has bad reduction at $p$ and the local model is discussed in this article. In a subsequent work we will define and study special cycles on this Shimura variety. 

In general, to describe the supersingular locus of a PEL-type Shimura variety, one can pass to the associated moduli space of $p$-divisible groups then use the uniformization theorem of Rapoport-Zink \cite{RZ-Aoms}. These moduli spaces of $p$-divisible groups are known as the basic Rapoport-Zink spaces and we are interested in calculating their special fibers. By passing from $p$-divisible groups to their assciated Dieudonn\.{e} modules, one obtain lattices in an isocrystal. By comparing the relative position between these lattices and the lattices representing the faces of the base alcove of an inner form of the underlying group of the PEL-problem, once can partite the Rapoport-Zink space into pieces which are related to classical Deligne-Lusztig varieties. This decomposition is known as the \emph{Bruhat-Tits stratification}. The terminology come from the pioneering work of \cite{Vol-can10} where the case of the Rapoport-Zink space of $\GU(1, n-1)$ at an inert prime is analyzed. The analysis is completed in \cite{VW-invent11}. This program has been extended to cover other PEL-type problems by many authors. We provide a list of works of this type and acknowledge their influences on this article.
\begin{itemize}
\item[-] For $\GU(1, n-1)$, $p$-inert and hyperspecial level, this is the work of \cite{Vol-can10} and \cite{VW-invent11}.  In a recent preprint \cite{Cho}, Sungyoon Cho is able to treat many cases of parahoric level structures.
\item[-] For $\GU(1, n-1)$, $p$-ramified and with level structure related to self-dual lattices, this is the work of \cite{RTW-MZ14}. When the level structure is related to a special parahoric and the Rapoport-Zink space admits the so called exotic good reduction, this is the work of \cite{Wu-thesis16}.
\item[-] For $\GU(2,2)$, $p$-split, this is considered by the author in \cite{Wang-2}. For $\GU(2,2)$, $p$-inert, this is done in \cite{HP14} by transferring the problem to $\GSpin(4,2)$. In the process of preparing of this note, we find the methods in this note can be also used to deal with $\GU(2,2)$, $p$-inert directly and this is documented in  \cite{Wang-2}.
\item[-] For $\GSpin(n-2,2)$ with hyperspecial level, this is the work of \cite{HP17}. In this case, the Shimura variety is not of PEL-type in general except for small ranks case where exceptional isomorphism happens. For example, $\GU(2,2)$ is essentially $\GSpin(4,2)$ and $\GSp(4)$ is essentially $\GSpin(3,2)$. The two Rapoport-Zink space studied in this note are both of non-hypersepcial level and therefore is not covered in \cite{HP17}, however they are all $\GSpin(3,2)$-type Shimura varieties.
\end{itemize}

A general understanding of the Bruhat-Tits stratification is achieved in the powerful work of \cite{GH-Cam15} and the subsequent work of \cite{GHN16}. There the problem is studied in the setting of affine Deligne-Lusztig varieties and a general group theoretic method is employed. Their work not only classifies which types of Shimura varieties admit Bruhat-Tits stratification but also gives an algorithm to compute the Deligne-Lusztig varieties occurring in the strata. The types of Shimura varieties in the above list as well as the two Shimura varieties in this article are covered by their work. A comparison between our results with \cite{GH-Cam15} is contained in the very last section of this article. 

Besides the Bruhat-Tits stratification, there are other methods to study the supersingular locus. We will only mention the works  \cite{Hel-Duk10}, \cite{HTX17} and \cite{CV-Alg18}. The method used in \cite{Hel-Duk10}, \cite{HTX17} is the so called isogeny trick and the isogeny is referred to the isogeny between the universal abelian varieties on different Shimura varieties. Notice that the isogeny trick used in Section $4$ in this note  is of same nature but in a local set up. The work of \cite{CV-Alg18} introduced the so called $J$-stratification for the affine Deligne-Lusztig varieties and is closely related to Bruhat-Tits stratification. In fact it can be shown that the $J$-stratification agrees with the Bruhat-Tits stratification for those classified by \cite{GH-Cam15}, see \cite{Gor18}.

\subsection{Results on Rapoport-Zink spaces}\label{RZ-def} We now introduce some notations and state the local results proved in this article.  Let $p$ be an odd prime and let $\FF$ be an algebraically closed field containing $\FF_{p}$. Let $W_{0}=W(\FF)$ be the Witt ring of $\FF$ and $K_{0}=W(\FF)_{\QQ}$ its fraction field. Let $B$ be a quaternion division algebra over $\QQ$ which ramifies at $p$ and splits at $\infty$. Denote by $*$ a neben involution on $B$, see \cite[A.4]{KR-ASENS94}. We fix a maximal order $\calO_{B}$ that is stable under $*$.  Let $N$ be a height $8$ isocrystal of slope $\frac{1}{2}$ equipped with a map $\iota: B_{p}\rightarrow \End({N})$ and an alternating form $(\cdot,\cdot): N\times N \rightarrow K_{0}$. We are concerned with the following moduli space for $p$-divisible groups with additional structures. 
Let $\Nilp$ be the category of $W_{0}$-schemes on which $p$ is locally nilpotent. We fix a  $p$-divisible group $\XX$ whose associated isocrystal is $N$ and a polarization $\lambda: \XX\rightarrow \XX^{\vee}$ corresponding to $(\cdot,\cdot)$ on $N$. We consider the set valued functor $\mathcal{N}$ that sends $S\in\Nilp$ to the isomorphism classes of the collection $(X, \iota_{X},  \lambda_{X}, \rho_{X})$ where
\begin{itemize} 
\item[-] $X$ is a $p$-divisible group of dimension $4$ and height $8$ over $S$;
\item[-]  $\lambda_{X}: X\rightarrow X^{\vee}$ is a principal polarization;
\item[-] $\iota_{X}: \calO_{B_{p}}\rightarrow \End_{S}(X)$ is an action of $\calO_{B_{p}}$ on $X$ defined over $S$;
\item[-]   $\rho_{X}: X\times_{S} S_{0}\rightarrow \XX\times_{\FF}S_{0}$ is an $\calO_{B_{p}}$-linear quasi-isogeny where $S_{0}$ is the special fiber of $S$ at $p$.
\end{itemize}

We require that $\iota_{X}$ satisfies the Kottwitz condition 
\begin{equation}\label{Kottwitz}
\det(T-\iota(c);\Lie(X))=(T^{2}-\mathrm{Trd}^{0}(c)T+ \mathrm{Nrd}^{0}(c))^{2}
\end{equation}
for $c\in\mathcal{O}_{B_{p}}$, where $\mathrm{Trd}^{0}(c)$ is the reduced trace of $c$ and $\mathrm{Nrd}^{0}(c)$ is the reduced norm of $c$. For $\rho_{X}: X\times_{S} S_{0}\rightarrow \XX\times_{\FF}S_{0}$, we require that 
\begin{equation}
\rho_{X}^{-1}\circ\lambda_{\XX}\circ\rho_{X}=c(\rho)\lambda_{X}
\end{equation}
for a $\QQ_{p}$-multiple $c(\rho)$.  
This moduli problem is representable by a formal scheme $\calN$, locally formally of finite type over $\Spf{W_{0}}$. The formal scheme $\calN$ can be decomposed into open and closed sub formal schemes $\calN=\bigsqcup_{i\in\ZZ}\calN(i)$ where each $\calN(i)$ is isomorphic to $\calN(0)$. Denote by $\calM=\calN_{red}(0)$ the underlying reduced scheme of $\calN(0)$. Then our main result concerns the structure of $\calM$.
\begin{thm}\label{intro-thm1}
The scheme $\calM$ can be decomposed into $\calM=\calM_{\{0\}}\cup\calM_{\{2\}}$. The irreducible components of $\calM_{\{0\}}$ and $\calM_{\{2\}}$ are all isomorphic to the surface defined by the equation $$x_{3}^{p}x_{0}-x^{p}_{0}x_{3}+x^{p}_{2}x_{1}-x^{p}_{1}x_{2}=0$$ where $x_{i}$ are the projective coordinates on $\PP^{3}$. If an irreducible component in $\calM_{\{0\}}$ intersect with an irreducible component of $\calM_{\{2\}}$ non-trivially, then the intersection is isomorphic to $\PP^{1}$. If an irreducible component in $\calM_{\{0\}}$ intersect with an other irreducible component of $\calM_{\{0\}}$ non-trivially, then the intersection is a point which is superspecial.  If an irreducible component in $\calM_{\{2\}}$ intersect with another irreducible component of $\calM_{\{2\}}$ non-trivially, then the intersection is a point which is superspecial.
\end{thm}

The decomposition $\calM=\calM_{\{0\}}\cup\calM_{\{2\}}$ in Theorem \ref{intro-thm1} can be made finer and we have the \emph{Bruhat-Tits stratification} of $\calM$.

\begin{thm}\label{intro-thm2}
We have the Bruhat-Tits stratification $$\calM=\calM^{\circ}_{\{0\}}\sqcup\calM^{\circ}_{\{2\}}\sqcup\calM^{\circ}_{\{0,2\}}\sqcup \calM_{\{1\}}.$$ The stratum $\calM^{\circ}_{\{0\}}$ is open an dense in $\calM_{\{0\}}$  and the complement is precisely $\calM^{\circ}_{\{0,2\}}\sqcup \calM_{\{1\}}$. Similarly the stratum $\calM^{\circ}_{\{2\}}$ is open and dense in $\calM_{\{2\}}$  and the complement is precisely $$\calM^{\circ}_{\{0,2\}}\sqcup \calM_{\{1\}}.$$ The irreducible components of $\calM^{\circ}_{\{0\}}$, $\calM^{\circ}_{\{2\}}$ and $\calM^{\circ}_{\{0,2\}}$ can be identified with classical Deligne-Lusztig varieties.
\end{thm}

We explain the connection of this stratification with Bruhat-Tits building. The quaternionic unitary Rapoport-Zink space is acted on by the split symplectic group of $4$-variables. If we label the vertices of a base alcove of the Bruhat-Tits building of $\Sp(4)$ by $\{0,1,2\}$, then the strata $\calM^{\circ}_{\{0\}}$, $\calM^{\circ}_{\{2\}}$, $\calM_{\{1\}}$ and $\calM^{\circ}_{\{0,2\}}$ correspond to precisely the vertices $0, 2, 1$ and the face $\{0,2\}$. The intersection behaviour is also determined by the incidence relation among faces of the Bruhat-Tits building. We refer the reader to Section \ref{main-result-section} for the details.

The quaternionic unitary Rapoport-Zink space in this article is closely related to the Rapoport-Zink space for $\GSp(4)$. Their relation is formally similar to the relation between the Lubin-Tate spaces and the Drinfeld upper-half spaces. But from a group theoretic point view or from the point view of local models, the quaternionic unitary Rapoport-Zink space is closer to the Rapoport-Zink space for $\GSp(4)$ with the so called paramodular level structure instead of with the hyperspecial level. Therefore we also include the calculation for the Bruhat-Tits stratification of the Rapoport-Zink space for $\GSp(4)$ with the paramodular level. However we should point out the description of the supersingular locus for the corresponding Shimura variety in the paramodular case is already known, we refer the reader to \cite{Yu-doc06} and \cite{Yu-Proc11} for the statement and proof of this result. The calculations in \cite{Yu-doc06} and \cite{Yu-Proc11} do not involve Rapoport-Zink space or Bruhat-Tits building and are indirect in the sense that they involve Shimura varieties with different level structures. Our calculation is direct and it is interesting to compare the two cases.

\subsection{Applications to the supersingular locus} We consider the following moduli problem $\gothS h_{U^{p}}$ in \cite{KR-ASENS94}. To a scheme $S$ over $\ZZ_{(p)}$, we classify the set of isomorphism classes of the collection $\{(A, \iota , \lambda, \eta )\}$ where:
\begin{itemize}
\item[-]  $A$ is an abelian scheme of relative dimension $4$ over $S$;
\item[-]  $\lambda: A\rightarrow A^{\vee}$ is a principal polarization which is $\calO_{B}$-linear;
\item[-]  $\iota: \calO_{B} \rightarrow \End_{S}(A)$ is a morphism such that $\lambda\circ i(a^{*})= i(a)^{\vee}\circ\lambda$ and satisfies the Kottwitz condition 
\begin{equation}\label{Ko-cond}
\det(T-\iota(a), \Lie(A))=(T^{2}-\mathrm{Trd}^{0}(a)T+\mathrm{Nrd}^{0}(a))^{2}
\end{equation} 
for all $a\in \calO_{B}$ and $\mathrm{Nrd}^{0}$ is the reduced norm of $B$ and $\mathrm{Trd}^{0}$ is the reduced trace of $B$;
\item[-]  $\eta$ is a suitable prime to $p$ level structure.
\end{itemize}

This moduli problem is representable by a quasi-projective scheme. We are interested in the closed subscheme  $Sh^{ss}_{U^{p}}$ in $\Sh_{U^{p}}\otimes \FF$ known as the supersingular locus which we view as a reduced scheme. The uniformization theorem of Rapoport-Zink \cite[Theorem 6.1]{RZ-Aoms} transfers the problem of describing the supersingular locus to that of describing the Rapoport-Zink space. More precisely there is an isomorphism of $\FF$ schemes 
$$Sh^{ss}_{U^{p}}\cong I(\QQ)\backslash \calN_{red}\times  G(\AAA^{p}_{f})/U^{p}.$$ 
where $I$ is an inner form of $G=\GU_{B}(2)$.  

\begin{thm}
The scheme $Sh^{ss}_{U^{p}}$ is pure of dimension $2$. For $U^{p}$ sufficiently small, the irreducible components are isomorphic to the surface $x_{3}^{p}x_{0}-x^{p}_{0}x_{3}+x^{p}_{2}x_{1}-x^{p}_{1}x_{2}=0$. The intersection of two irreducible components if nonempty is either isomorphic to a $\PP^{1}$ or is a  point.
\end{thm}

We end this introduction by mentioning two applications of the results that are obtained in this article.
\begin{itemize}
\item[-] In the preprint \cite{Wang19a}, we use  the results in this article to describe the quaternionic unitary Rapoport-Zink space with arbitary parahoric level structure. 
\item[-] We will also apply the results in this article to the questions of level lowering for Siegel modular forms improving the results in \cite{VH19}, see the forthcoming work \cite{Wang19b}. 
\end{itemize}

\subsection{Acknowledgement} The author would like to thank Henri Darmon for supporting his postdoctoral studies. He would like to thank Liang Xiao for many helpful conversations regarding to Shimura varieties and beyond. He is grateful to Ulrich G\"{o}rtz for all the help and his comments on this article. He is inspired by reading many works of Chia-Fu Yu on Siegel modular varieties. He also would like to thank Ben Howard, Eyal Goren, Yichao Tian, Xu Shen  and Benedict Gross for valuable discussions related to this article.  He would like to thank the referees for all the corrections.

While this article is being reviewed, Yasuhiro Oki obtained similar results about the supersingular locus of the quaternionic unitary Shimura variety independently. His method is completely different. He exploited the exceptional isomorphisms between the group $\GU_{B}(2)$ and the non-split $\GSpin(3,2)$. Then he embedded the non-split $\GSpin(3,2)$ in the split $\GSpin(4,2)$. This allows him to use the results of Howard-Pappas \cite{HP14} mentioned before. We would like to thank Yoichi Mieda for sending Oki's work to us. 

\subsection{Notations} Let $p$ be an odd prime and let $\FF$ be an algebraically closed field containing $\FF_{p}$. Let $W_{0}=W(\FF)$ be the Witt ring of $\FF$ and $K_{0}=W(\FF)_{\QQ}$ be its fraction field. Denote by $\psi_{0}$ and $\psi_{1}$ the two embeddings of $\FF_{p^{2}}$ in $\FF$.  Let $\sigma$ be the Frobenius on $\FF$. Let $M_{1}\subset M_{2}$ be two $W(\FF)$-modules we wrire $M_{1}\subset^{d} M_{2}$ if the colength of the inclusion is $d$. If $R$ is ring and $L$ is an $R$-module and $R^{\prime}$ is an $R$-algebra, we define $L_{R^{\prime}}=L\otimes_{R} R^{\prime}$. 
Let $B$ be a quaternion division algebra over $\QQ$ which ramifies at $p$ and splits at $\infty$.

\section{Quaternionic unitary Rapoport-Zink space}
\subsection{Isocrystal with quaternionic multiplication}\label{quaternionic-module} Let $N$ be an isocrystal of height $8$ over $\FF$ which is isotypic of slope $\frac{1}{2}$. Let $\iota: B_{p}\rightarrow \End({N})$ be an action of $B_{p}$ on $N$. We can write $B_{p}$ as $\QQ_{p^{2}}+\QQ_{p^{2}}\Pi$ with $\Pi^{2}=p$. Suppose $N$ is equipped with an alternating form $(\cdot,\cdot): N\times N \rightarrow K_{0}$ with the property that $(F x, y)= (x, V y)^{\sigma}$ and $(bx, y)= (x, b^{*}y)$ where $b^{*}$ is the image of $b$ under a neben involution of the quaternion algebra. More precisely, the involution can be given by 
\begin{equation*}
\begin{split}
&x^{*}=\sigma(x), x\in \QQ_{p^{2}}\\
&\Pi^{*}=\Pi.\\
\end{split}
\end{equation*}
The action of $\mathcal{O}_{B_{p}}$ decomposes $N$ into $N=N_{0}\oplus N_{1}$. Restricting the symplectic form $(x, y)_{0}:= ( x, \Pi y)$ to $N_{0}$ gives an identification between the group $\GU_{B_{p}}(N)$ and $\GSp(N_{0})$
over $W_{0}$ cf. \cite[1.42]{RZ-Aoms}.

We will use covariant Dieudonn\'{e} theory throughout this article. A Dieudonn\'{e} lattice $M$ is a lattice in $N$  with the property that $pM \subset FM\subset M$.  A Dieudonn\'{e} lattice is \emph{superspecial} if $F^{2}M=pM$ and in this case $F=V$. We are concerned with Diedonne\'{e} lattice with an additional endomorphism $\iota: \mathcal{O}_{B_{p}}\rightarrow \End(M)$. Here we can represent $\mathcal{O}_{B_{p}}$ by $\ZZ_{p^{2}}+\ZZ_{p^{2}}\Pi$ and hence we can decompose M as $M_{0}\oplus M_{1}$ with an additional operator $\Pi$ that swaps the two components. The alternating form $(\cdot,\cdot)$, restricted to $M$, induces a pairing $(\cdot,\cdot): M_{0}\times M_{1}\rightarrow W_{0}$.

\subsection{Rapoport-Zink space for quaternionic unitary group} 
Recall the set valued functor $\mathcal{N}$ defined in Section \ref{RZ-def}. Let $S\in \Nilp$. The $S$-valued points $\mathcal{N}(S)$ of $\mathcal{N}$ are given by the set of isomorphism classes of the quadruples $(X, \iota_{X}, \lambda_{X}, \rho_{X})$ where

\begin{itemize} 
\item[-] $X$ is a $p$-divisible group of dimension $4$ and height $8$;
\item[-] $\lambda_{X}: X\rightarrow X^{\vee}$ is a principal polarization; 
\item[-] $\iota_{X}: \calO_{B_{p}}\rightarrow \End_{S}(X)$ is an action of $\calO_{B_{p}}$ satisfying the Kottwitz condition \eqref{Kottwitz};
\item[-] $\rho_{X}: X\times_{S} S_{0}\rightarrow \XX\times_{\FF}S_{0}$ is an $\calO_{B_{p}}$-linear quasi-isogeny .
\end{itemize}

\begin{lemma}\label{singularity}
The functor $\calN$ is representable by a separated formal scheme locally formally of finite type over $\Spf{W_{0}}$. Moreover $\calN$ is flat and formally smooth away from the superspecial point.
\end{lemma}

\begin{proof}
The representability follows from \cite{RZ-Aoms} and the second assertion on flatness and the singularity relies on the analysis of the local model below. We will see after unramified base change it is equivalent to the local model for the split symplectic group of dimension $4$ with paramodular level.  Then the result we need here is proved in \cite{Yu-Proc11}. See the section below for a discussion.
\end{proof}

\subsection{Local model}The local model for this PEL-type Rapoport-Zink space is defined as in \cite[Definition 3.27]{RZ-Aoms}. Let ${\bf{M}}^{\loc}$ be the set valued functor on the category $\ZZ_{p}$-schemes defined as follows. Let $S$ be a test scheme then the set ${\bf{M}}^{\loc}(S)$ is the set of $(\Lambda, \calF_{\Lambda})$ up to ismorphism where 
\begin{itemize}
\item[-] $\Lambda$ is a $\calO_{B_{p}}$-module of $\ZZ_{p}$-rank 8;
\item[-] $\calF_{\Lambda}\subset \Lambda\otimes\calO_{S}$ is a locally free $\calO_{B_{p}}\otimes_{\ZZ_{p}} \calO_{S}$ submodule of $\ZZ_{p}$-rank $4$. 
\end{itemize}
The module $\Lambda$ is equipped with a alternating form $(\cdot,\cdot)$ that satisfies $(bx, y)=(x, b^{*}y)$ and we require that $\Lambda\cong \Lambda^{\perp}$ where $\Lambda^{\perp}$ is the integral dual of $\Lambda$ with respect to the form $(\cdot,\cdot)$ and that $\calF_{\Lambda}$ is isotropic. We also require that $\calF_{\Lambda}$ satisfies the Kottwitz condition: $$\det(T-\iota(a); \calF_{\Lambda})=(T^{2}-\mathrm{Trd}^{0}(a)T+\mathrm{Nrd}^{0}(a))^{2}$$ for every $a\in \calO_{B_{p}}$. As we have seen in Section \ref{quaternionic-module},  for each $(\Lambda, \calF_{\Lambda})$ we have a decomposition $\Lambda\otimes \ZZ_{p^{2}}=\Lambda_{0}\oplus\Lambda_{1}\supset \calF_{\Lambda}\otimes \ZZ_{p^{2}}=\calF_{0}\oplus\calF_{1}$. The conditions we put on $(\Lambda, \calF_{\Lambda})$ translates to the conditions that $\Lambda^{\perp}_{0}=\Lambda_{1}$, $\Pi \Lambda_{0}\subset^{2} \Lambda_{1}$, $\Pi\Lambda_{1}\subset^{2} \Lambda_{0}$, $\Pi\calF_{0}\subset \calF_{1}$ and $\Pi\calF_{1}\subset \calF_{0}$. We define a new alternating form on $\Lambda_{0}$ by $(x,y)_{0}=(x, \Pi y)$. With respect to this new form, we see that the pair $(\Lambda, \calF_{\Lambda})$ is equivalent to the pair $(\Lambda_{0}, \calF_{0})$ that satisfies the condition $p\Lambda^{\vee}_{0}\subset \Lambda_{0}\subset \Lambda^{\vee}_{0}$, $\calF_{0}$ is isotropic with respect to $(\cdot, \cdot)_{0}$.  For example the first condition can be seen from the following computation,
\begin{align*}\label{tau-dual}
 \Pi \Lambda_{1}= \Pi \Lambda_{0}^{\perp} &=\{  x\in \Lambda_{0}\otimes \QQ_{p}: ( \Lambda_{0}, \Pi^{-1} x)\subset \ZZ_{p^{2}} \}\\
     &=\{x: (\Lambda_{0}, \frac{1}{p} \Pi x)\subset \ZZ_{p^{2}} \}\\
     &=p \Lambda^{\vee}_{0}.
\end{align*}
Hence after base change to $W_{0}$, we see that ${\bf{M}}^{\loc}_{W_{0}}$ is equivalent to the functor ${\bf{M}}^{\loc}_{\{1\}, W_{0}}$ where ${\bf{M}}^{\loc}_{\{1\}}$ is the set valued functor on the category of $\ZZ_{p}$-schemes defined as follows. For $S$ a test $\ZZ_{p}$-scheme,  the set ${\bf{M}}^{\loc}_{\{1\}}(S)$ classifies the isomorphism classes of the pairs $(\Lambda, \calF_{\Lambda})$ where 
\begin{itemize}
\item[-] $\Lambda$ is a $\ZZ_{p}$-module of rank $4$;
\item[-] $\calF_{\Lambda}\subset \Lambda\otimes \calO_{S}$ is locally free of rank $2$. 
\end{itemize}
The lattice $\Lambda$ is equipped with an alternating form $(\cdot,\cdot)_{0}$ and we require that $$p\Lambda^{\vee}\subset^{2} \Lambda \subset^{2} \Lambda^{\vee}$$ for the integral dual $\Lambda^{\vee}$ of $\Lambda$ with respect to the form $(\cdot,\cdot)_{0}$ and $\calF_{\Lambda}$ is isotropic with respect to $(\cdot, \cdot)_{0}$. The latter local model ${\bf{M}}^{\loc}_{\{1\}}$ arises when one studies the Siegel threefold with paramodular level structure and we refer the reader to the article  \cite{Yu-Proc11} for the study of this local model. In particular, by \cite[Lemma 2.4]{Yu-Proc11}, the singularity of the special of ${\bf{M}}^{\loc}_{\{1\}}$ is reflected by the condition that $p\Lambda^{\vee}_{0}=\calF_{0}$. Then using the local model diagram \cite[3.29]{RZ-Aoms}, this translates to the condition that $VM_{1}=\Pi M_{1}$ for a Dieudonn\'{e} lattice $M=M_{0}\oplus M_{1}$ under the identification of $M$ to $\Lambda_{W_{0}}$. One then concludes that $M$ is superspeical easily from this. This finishes the proof of Lemma \ref{singularity}.  The following theorem is \cite[ Theorem 1.3]{Yu-Proc11} which characterize the type of singularity of the local model of ${\bf{M}}^{\loc}_{\{1\}, W_{0}}$ and hence of ${\bf{M}}^{\loc}_{W_{0}}$.

\begin{theorem}
The special fiber of the local model ${\bf{M}}^{\loc}_{\{1\}, W_{0}}$ is singular in a discrete set of points and the formal completion of the local ring at a singular point in the special fiber is given by $$\FF[[x_{11}, x_{12}, x_{21}, x_{22}]]/(x_{11}x_{22}-x_{12}x_{21}).$$
\end{theorem}

\subsection{Points of the Rapoport-Zink space}By passing to Dieudonn\'{e} modules, the $\FF$-points of the Rapoport-Zink space $\mathcal{N}(\FF)$ can be identified with the following set
\begin{equation}\label{set1}
\begin{split}
&\mathcal{N}(\FF) =\{M\subset N:M^{\perp}=cM, pM\subset VM\subset M, \\
&\Pi M_{0}\subset^{2} M_{1}, \Pi M_{1}\subset^{2} M_{0}, V M_{1}\subset^{2} M_{0}, V M_{0}\subset^{2} M_{1}
\}.\\ 
\end{split}
\end{equation} 
The last two conditions of the definition follows from the Kottwitz condition and similar condition on $\Pi$ follows from the following commutative diagram:

$$
\begin{tikzcd}
  M_{0} \arrow[r,  "V"] \arrow[d,  "\Pi"]
    & M_{1} \arrow[d,  "\Pi"]  \\
  M_{1} \arrow[r,  "V"]
&M_{0} .
\end{tikzcd}
$$ 

We claim the following Pappas condition on $\calN$  is automatic:
\begin{equation}\label{Pappas_condition}
\wedge^{2}(\iota(\Pi); \Lie(X)_{i})=0 \text{  for  } i=0,1.
\end{equation}
Indeed, since we have the composite of $\Pi: M_{0}/VM_{1}\rightarrow M_{1}/VM_{0}$ and $\Pi: M_{1}/VM_{0}\rightarrow M_{0}/VM_{1}$ being multiplication by $p$, at least one of them is not invertible. Note also that we have $(\Pi x, y )=(x, \Pi y)$, the other is also not invertible.

On the level of $\FF$-points, this translates to 
\begin{equation}\label{spin}
\dim_{\FF} VM_{0}+\Pi M_{0}/VM_{0} \leq1 \text{ and }
\dim_{\FF} VM_{1}+\Pi M_{1}/VM_{1} \leq1.
\end{equation}
We can restrict to the component of $\mathcal{N}(0)$ that classifies degree $0$ quasi-isogenies in the definition of the Rapoport-Zink space i.e $c(\rho)\in\ZZ^{\times}_{p}$ and $M=M^{\perp}$ in $\mathcal{N}$ .  In light of what we discussed above that the identification between $\GU_{B_{p}}(2)$ and $\GSp(4)$ over $W_{0}$ is via restricting to the $0$-th component $N_{0}$, we will similarly restrict Dieudonn\'{e} lattices $M$ to their $0$-th componet $M_{0}$. We put $\tau:= \Pi V^{-1}$ and for a lattice $D\subset N_{0}$ we denote by $D^{\vee}$ the integral dual of $D$ in $N_{0}$ with respect to $(\cdot,\cdot)_{0}$. The set theoretic description in \eqref{set1} of the Rapoport-Zink space then can be reformulated in terms of lattices in $N_{0}$ with respect the form $(\cdot,\cdot)_{0}$.  We are interested in the underlying reduced scheme structure of $\calN(0)$ and we denote this reduced scheme by $\calM=\calN_{red}(0)$. The $\FF$-points of $\calM$ are given as follows.
 
 \begin{proposition}\label{set2}
 There is a bijection between $\calM(\FF)$ and the following set
 $$\mathcal{M}(\FF)=\{D\subset N_{0}:pD^{\vee}\subset^{2} D\subset^{2} D^{\vee},  pD^{\vee}\subset^{2} \tau(D)\subset^{2} D^{\vee} \}.$$
 The map is sending $M=M_{0}\oplus M_{1}$ to $D=M_{0}$. Moreover $\dim_{\FF}D+\tau (D)/D\leq1$. 
 \end{proposition}
 \begin{proof}
 The first condition corresponds to the condition $\Pi M_{0}\subset^{2} M_{1}$. The second condition corresponds to the Kottwitz condition \eqref{Kottwitz}. We only spell out the details for the second condition as similar computation is used throughout this article. Since $pM_{1}\subset FM_{0}\subset M_{1}$, we have $\Pi M_{1}\subset \Pi V^{-1}M_{0}=\tau(M_{0})\subset \Pi^{-1}M_{1}$.  Notice we have:
 \begin{align*}\label{tau-dual}
 \Pi M_{1}= \Pi M_{0}^{\perp} &=\{  x\in N_{0}: ( M_{0}, \Pi^{-1} x)\subset W_{0} \}\\
     &=\{x\in N_{0}: (M_{0}, \frac{1}{p} \Pi x)\subset W_{0} \}\\
     &=p M^{\vee}_{0}.
\end{align*}
The second condition follows. The moreover part is a direct translation of the Pappas condition \eqref{spin}.
\end{proof}

\section{Bruhat-Tits Stratification for quaternionic unitary group} 
\subsection{Deligne-Lusztig varieties}Let $G_{0}$ be a connected reductive group over $\FF_{p}$ and let $G=G_{0,\FF}$ be its base change to $\FF$. Let $T$ be a maximal torus contained in $G$ and $B$ a Borel subgroup containing $T$. We can assume $T$ is defined over $\FF_{p}$. Let $W$ be the finite Weyl group corresponding to $(T, B)$. Then $W$ affords an action by the Frobenius $\sigma$ induced by the Frobenius action on $G$.

Let $\Delta^{*}=\{\alpha_{1}, \cdots, \alpha_{n}\}$ be the simple roots corresponding to $(T, B)$ and let $s_{i}$ be the simple reflection corresponding to the 
root $\alpha_{i}$. For $I\subset \Delta^{*}$, let $W_{I}$ be the subgroup of $W$ generated by $\{s_{i}: i\in I\}$.  Consider $P_{I}$ the corresponding  parabolic subgroup of $G$.  For another set $J\subset \Delta^{*}$, we have a decomposition $$G=\bigcup_{w\in W_{I} \backslash W/ W_{J}}P_{I}wP_{J}$$ and hence a bijection
$$P_{I}\backslash G/ P_{J} \cong W_{I} \backslash W/ W_{J}.$$
We define the relative position map $$\text{inv}: G/P_{I}\times G/P_{J}\rightarrow W_{I} \backslash W/ W_{J} $$ by sending $(g_{1}, g_{2})$ to $g^{-1}_{1}g_{2}\in P_{I}\backslash G/ P_{J} $. 

\begin{definition}
For a given $w\in W_{I}\backslash W/ W_{\sigma(I)}$, the Deligne-Lusztig variety $X_{P_{I}}(w)$ is the locally closed reduced subscheme of $G/P_{I}$ whose $\FF$-points are described by the set 
\begin{equation}
X_{P_{I}}(w)=\{gP_{I}\in G/P_{I}; \mathrm{inv}(g, \sigma(g))=w\}.
\end{equation}
\end{definition}

\subsection{The symplectic group}Let $G_{0}$ be the symplectic group over $\FF_{p}$ defined by a symplectic space $(V, (\cdot,\cdot))$ of dimension 4 over $\FF_{p}$. We choose a basis $(e_{1}\cdots e_{4})$ of $V$ that $(e_{i}, e_{5-j})=\pm \delta_{i,j}$ for $i=1,\cdots, 4$. The simple reflections in $W$ can be understood as follows, as elements in $S_{4}$,
\begin{equation}\label{w_1}
\begin{split}
&s_{1}=(12)(34);\\
&s_{2}=(23);\\
&s_{1}s_{2}=(1243).\\
\end{split}
\end{equation}
We define the elements $w_{1}, w_{2}\in W$  by $w_{1}=s_{1}$ and $w_{2}=s_{1}s_{2}$.

Denote by $G$ the group $G=G_{0}\otimes \FF$. We are concerned with the  Deligne-Lusztig variety $X_{P_{I}}(w)$ for $I=\{1\}$ where $P_{\{1\}}$ is the Klingen parabolic.  Consider the following closed subvarieties of $G/P_{\{1\}}$

\begin{equation}\label{DL_vertex_2}
Y^{(+)}_{V}=\{U\subset V_{\FF}: \dim_{\FF}U=1, U\subset U^{\perp}, \dim_{\FF}U+ \sigma(U)/U\leq 1,  U\subset U^{\perp}\cap \sigma(U^{\perp})\};
\end{equation}
and 
\begin{equation}\label{DL_vertex_0}
Y^{(-)}_{V}=\{U\subset V_{\FF}: \dim_{\FF}U=3, U^{\perp}\subset U, \dim_{\FF}U /U\cap\sigma(U)\leq 1, U^{\perp}\subset U\cap \sigma(U)\}.
\end{equation}
They are obviously isomorphic to each other. We would like to define a stratification on $Y^{(\pm)}_{V}$.

\begin{theorem}\label{DL-stratum}
There is a decomposition of $Y^{(\pm)}_{V}$ into disjoint union of locally closed subvarieties
\begin{equation}
Y^{(\pm)}_{V}=X_{P_{\{1\}}}(1)\sqcup X_{B}(w_{1})\sqcup X_{B}(w_{2})
\end{equation}
where $w_{1}$ and $w_{2}$ are as in \eqref{w_1}. Here $X_{P_{\{1\}}}(1)$ is closed and of dimension $0$, $X_{B}(w_{1})$ is of dimension $1$ and  $X_{B}(w_{2})$ is two dimensional. Moreover $Y^{(\pm)}_{V}$ is the closure of $X_{B}(w_{2})$.
\end{theorem}

\begin{proof}
We use the description of $Y^{(-)}_{V}$ as in \eqref{DL_vertex_0}.
One has $Y^{(-)}_{V}=X_{P_{\{1\}}}(1)\cup X_{P_{\{1\}}}(w_{1})$ and $X_{P_{1}}(w_{1})=X_{B}(w_{1})\cup X_{B}(w_{2})$. In fact using the description as \ref{DL_vertex_0}, one has the following descriptions
\begin{itemize}
\item[-] $X_{P_{\{1\}}}(1)$ consists of $U$ that is $\sigma$-stable;
\item[-] $X_{B}(w_{1})$ consists of $U$ that is not $\sigma$-stable but $U\cap \sigma(U)$ is $\sigma$-stable and is a totally isotropic plane;
\item[-] $X_{B}(w_{2})$ consists of $U$ that is not $\sigma$-stable and $U\cap \sigma(U)$ is not $\sigma$-stable but is a totally isotropic plane.
\end{itemize}

There is a more concrete description of the each stratum. The closed stratum $X_{P_{\{1\}}}(1)$ consists of a discrete set of $\FF_{p}$-rational points of $Y^{(-)}_{V}(\FF)$.

The irreducible components of $X_{B}(w_{1})$ are the complement of $\FF_{p}$-points in $\PP^{1}$. Indeed, since $U$ is not $\sigma$-stable, $U\cap\sigma(U)$ is of dimension $2$. Moreover $U\cap \sigma(U)$ has the property that $U\cap\sigma(U)= U^{\perp}+ \sigma(U)^{\perp}=(U\cap \sigma(U))^{\perp}$. The line $U^{\perp}\subset U\cap\sigma(U)$ gives a point in $\PP^{1}$. This point is not $\FF_{p}$-rational as $U$ is not $\sigma$-stable. Conversely, given any $\sigma$-stable plane $T$ in $V$ with the property that $T=T^{\perp}$ and a line $U^{\perp}\subset T$, $U$ gives a point in $X_{B}(w_{1})$. Indeed, $U^{\perp}\subset T=T^{\perp}\subset U$ and $U\cap \sigma(U)=T$ follows by considering the index among the vector spaces. 

The open stratum $X_{B}(w_{2})$ is the surface defined as in  \cite[2.4]{DL-Ann76}. It is the surface defined by  the set of $x\in\PP(V)$  such that $(x, \sigma(x))=0$ and $(x,\sigma^{2}(x))\neq 0$. By choosing a coordinate system on $\PP(V)$, say $x_{0}, x_{1}, x_{2}, x_{3}$, we see it is defined by removing from the surface $$x_{3}^{p}x_{0}-x^{p}_{0}x_{3}+x^{p}_{2}x_{1}-x^{p}_{1}x_{2}=0$$ a collection of $\PP^{1}$, one for each isotropic plane defined over $\FF_{p}$.
The closure of $X_{B}(w_{2})$ is $Y_{V}$ and it is the surface defined by $x_{3}^{p}x_{0}-x^{p}_{0}x_{3}+x^{p}_{2}x_{1}-x^{p}_{1}x_{2}=0$.

\end{proof}

\subsection {Vertex lattices in $\GSp(4)$}\label{vetex-lattices}
The affine Dynkin diagram of type $\tilde{C}_{2}$ is given by
\begin{displaymath}
  \xymatrix{\underset{0}\bullet \ar@2{->}[r] &\underset{1}\bullet &\underset{2}\bullet\ar@2{->}[l]}
\end{displaymath}
where the nodes $0$ and $2$ are special and $1$ is not special in the sense that the parahoric subgroup corresponding to $0$ and $2$ are special parahoric subgroups.

Let $V$ be symplectic space over $\QQ_{p}$ of dimension $4$. The vertices of a base alcove of the building for $\Sp(4)(V)$ is identified with the nodes in the Dynkin diagram. Thus we label the vertices by the set $\{0, 1, 2\}$. The vertex $0$ corresponds to a lattice $L$ with $pL^{\vee}\subset^{4} L \subset^{0} L^{\vee}$, the vertex $1$ corresponds to a lattice $L$ with $pL^{\vee}\subset^{2} L \subset^{2} L^{\vee}$ and the vertex $2$ corresponds to a lattice $L$ with $pL^{\vee}\subset^{0} L \subset^{4} L^{\vee}$. We restrict the form $(\cdot,\cdot)_{0}$ to $C$ and consider $C$ as a symplectic space. The vertices of the base alcove can also be identified with the maximal parahorics of the group $\GSp(4)(\QQ_{p})$. The vertices $0$ and $2$ both correspond to the \emph{hyperspecial parahoric} and vertex $1$ corresponds to the \emph{paramodular parahoric}. The edge connecting the vertices $0$ and $2$ corresponds to the \emph{Siegel parahoric} which is not maximal.

Let $b$ be an element in its $\sigma$-conjugacy class which gives the isocrystal $N$ and let $J_{b}$ be the $\sigma$-conjugate centralizer of $b$ in the group $\GU_{B_{p}}(N)$, see \cite[Proposition 1.12]{RZ-Aoms}. Consider the slope zero isocystal $(N_{0}, \tau)$, by a theorem of Dieudonn\'{e} it gives rise to a factorization $N_{0}=C\otimes_{\QQ_{p}} K_{0}$ where $C=N^{\tau=1}_{0}$. We restrict the form $(\cdot,\cdot)_{0}$ to $C$ and thus consider $C$ as a symplectic space over $\QQ_{p}$.

\begin{lemma}
There is an isomorphism between the group $J_{b}(\QQ_{p})$ the $\sigma$-conjugate centralizer of $b$ which gives the isocystal $N$ and the group $\GSp(C)$. 
\end{lemma}
\begin{proof}
By definition, we have $J_{b}=\{g\in \GU_{B_{p}}(N); gF=Fg\}$. Since $g$ commutes with the action of $B_{p}$, it respects the decomposition of $N=N_{0}\oplus N_{1}$.  Since it also commutes with the Frobenius $F$, the action is uniquely determined on its restriction to $N_{0}$. The result then follows from the fact that $g$ also commutes with $\tau$. This claim is also stated in \cite[1.42]{RZ-Aoms}. 
\end{proof}

\begin{definition}
We consider $\ZZ_{p}$-lattices $L$ in $C$.  We call such a lattice $L$ a vertex lattice of type $i$ if it represents a vertex of type $i$ in the fixed base alcove of the Bruhat-Tits building of $\Sp(C)$ for $i\in\{0,1,2\}$.
\end{definition}

Next we associate to each vertex lattice $L$ a characteristic subset $\mathcal{M}_{L}(\FF)$ of $\mathcal{M}(\FF)$.

\begin{definition}\label{stratum-set}
\hfill
\begin{enumerate}
\item For a vertex lattice $L_{0}$ of type $0$, the stratum $\mathcal{M}_{L_{0}}(\FF)$ is the set $$ \{D\in \mathcal{M}(\FF): D\subset L_{0,W_{0}}\}.$$
\item For a vertex lattice $L_{2}$ of type $2$, the stratum $\mathcal{M}_{L_{2}}(\FF)$ is the set $$ \{D\in \mathcal{M}(\FF): L_{2, W_{0}}\subset D\}.$$
\item For a vertex lattice $L_{1}$ of type $1$, the stratum $\mathcal{M}_{L_{1}}(\FF)$ is the set $$\{D\in \mathcal{M}(\FF): L_{1, W_{0}}\subset D\}.$$ 
\end{enumerate}
We refer to these sets $\calM_{L_{i}}(\FF)$ as lattice strata of $\calM(\FF)$. 
\end{definition}

To show these lattice strata give a partition of the set $\mathcal{M}(\FF)$,  we will rely on an analogue of the crucial lemma as in \cite[Lemma 2.1]{Vol-can10} .
\begin{proposition}\label{crucial_lemma}
Given $D\in \mathcal{M}(\FF)$, we can find a $\tau$-stable lattice $L(D)$ in $N_{0}$ that it either fits in the following chain of lattices
\begin{equation}\label{0-4type}pL(D)^{\vee}\subset pD^{\vee}\subset D\subset L(D)\subset L(D)^{\vee}\subset D^{\vee}\end{equation} 
or it fits in 
\begin{equation}\label{4-0type}pD^{\vee}\subset pL(D)^{\vee}\subset L(D)\subset D\subset D^{\vee}\subset L(D)^{\vee}.\end{equation}
\end{proposition}
\begin{proof}
The proof of this is similar to  \cite{Vol-can10} {Lemma 2.1} but there the Kottwitz condition is different.

Suppose $D$ is $\tau$-stable, then there is nothing to prove and we simply let $L(D)=D$.  Otherwise, the Pappas condition shows that $D\subset^{1} D+\tau(D)$. Suppose that $D+\tau(D)$ is $\tau$-stable, then we set $L(D)=D+\tau(D)$. Notice that Since $D\in \mathcal{M}(\FF)$, we have \begin{equation}\label{condition}pD^{\vee}\subset D\subset D^{\vee} \text{ and } pD^{\vee}\subset \tau(D)\subset D^{\vee}.\end{equation} Therefore $L(D)=D+\tau(D)\subset L(D)^{\vee}=D^{\vee}\cap \tau(D)^{\vee}$ and  $L(D)$ fits in \eqref{0-4type}.

Now suppose $D+\tau(D)$ is not $\tau$-stable and we claim that $D\cap \tau(D)$ is $\tau$-stable. Indeed, suppose otherwise $D\cap\tau(D)$ is not $\tau$-stable, by \eqref{condition}, we have $p\tau(D)^{\vee}\subset^{1} D\cap \tau(D)$ and $p\tau(D)^{\vee}\subset^{1} \tau(D)\cap\tau^{2}(D)$. Hence \begin{equation}\label{three-intersection}D\cap \tau(D)\cap \tau^{2}(D)=p\tau(D)^{\vee}.\end{equation} Now, we consider the family of lattices $L_{j}(D)=D+\tau(D)+\cdots +\tau^{j}(D)$. There is a minimal $d$ such that $L_{d}(D)=\tau(L_{d}(D))$ by \cite[Proposition 2.7]{RZ-Aoms}  and we set $L(D)=L_{d}(D)$.  From the Pappas condition in \ref{set2}, we deduce that $$\tau(L_{j-2}(D))\subset^{1} L_{j-1}(D)\subset^{1} L_{j}(D)$$ and $$\tau(L_{j-2}(D))\subset^{1} \tau(L_{j-1}(D))\subset^{1} L_{j}(D).$$ It follows that \begin{equation}\label{inter}\tau(L_{j-2}(D))=L_{j-1}(D)\cap \tau(L_{j-1}(D)).\end{equation} By \eqref{condition} $D+\tau(D)\subset D^{\vee}\subset p^{-1}\tau(D)$ and  $\tau(D)+\tau^{2}(D)\subset \tau(D)^{\vee}\subset p^{-1}\tau(D)$. Then we have $D+\tau(D)+\tau^{2}(D)\subset p^{-1}\tau(D)$. Thus $$L_{d}(D)=D+\tau(D)+\cdots +\tau^{d}(D)\subset p^{-1}\tau(D)+p^{-1}\tau^{2}(D)+\cdots p^{-1}\tau^{d-1}(D)\subset p^{-1}L_{d-1}(D).$$ 
Applying $\tau^{l}$ for any integer $l$ on both sides, we get $L_{d}(D)\subset p^{-1} \bigcap_{l\in \ZZ}\tau^{l}(L_{d-1}(D))$. It follows from \eqref{inter} that $$\bigcap_{l\in\ZZ}L_{d-1}(D)=\bigcap_{l\in\ZZ}L_{d-2}(D)=\cdots= \bigcap_{l\in \ZZ}D.$$ Hence $$L_{d}(D)\subset p^{-1} \bigcap_{l\in \ZZ}\tau^{l}(L_{d-1}(D))= p^{-1}\bigcap_{l\in \ZZ}\tau^{l}(D).$$ Now we apply \eqref{three-intersection} and we find $$L_{d}(D)\subset p^{-1} \bigcap_{l\in \ZZ}\tau^{l}(L_{d-1}(D))= p^{-1}\bigcap_{l\in \ZZ}\tau^{l}(D)=\bigcap_{l\in \ZZ} \tau^{l}(D)^{\vee}=L_{d}(D)^{\vee}.$$ Then $L_{d}(D)$ satisfies \eqref{0-4type} and it follows that the indices among the lattices has to be $pL_{d}(D)^{\vee}\subset^{1} pD^{\vee}\subset^{2} D\subset^{1} L(D)\subset^{0} L_{d}(D)^{\vee}\subset^{1} D^{\vee}$. This implies $L_{d}(D)=D+\tau(D)$ and this is a contradiction to the assumption that $D+\tau(D)$ is not $\tau$ stable. This proves in this case $D\cap \tau(D)$ is $\tau$-stable and we set $L(D)=D\cap\tau(D)$. Then one checks easily as before that $L(D)$ satisfies \eqref{4-0type}.

 \end{proof}

\begin{remark}
The lattice $L(D)$ in the previous proposition gives rise to a vertex lattice $L$ by $L=L(D)^{\tau=1}$.
\end{remark}

\begin{definition}
For $i=0,1,2$, we define the Bruhat-Tits stratum of type $i$ to be
 $$\calM_{\{i\}}(\FF)= \bigcup_{L_{i}}\calM_{L_{i}}(\FF)$$ where the union is taken over all the the vertex lattices of type $i$.
\end{definition}

\begin{corollary}
We have a decomposition of the set $\calM(\FF)$ by $$\mathcal{M}(\FF)=\mathcal{M}_{\{0\}}(\FF)\cup\mathcal{M}_{\{1\}}(\FF)\cup \mathcal{M}_{\{2\}}(\FF).$$ 
\end{corollary}
\begin{proof}
This follows from Proposition \ref{crucial_lemma}. Indeed, if $D$ is $\tau$-invariant, then we have $pL^{\vee}(D)\subset^{2} L(D)\subset^{2} L^{\vee}(D)$ and $D=L(D)$ gives a vertex lattice of type $1$ and $D\in \calM_{\{1\}}$.

If $D$ is not $\tau$-invariant, then $L(D)$ either fits in $$pL(D)^{\vee}\subset^{1} pD^{\vee}\subset^{2} D\subset^{1} L(D)\subset^{0} L(D)^{\vee}\subset^{1} D^{\vee}$$ or fits in $$pD^{\vee}\subset^{1} pL(D)^{\vee}\subset^{0} L(D)\subset^{1} D\subset^{2} D^{\vee}\subset^{1} L(D)^{\vee}.$$

In the first case $L(D)$ gives a vertex lattice of type $0$ and $D\in \calM_{\{0\}}$. In the second case $L(D)$ gives a vertex lattice of type $2$ and $D\in \calM_{\{2\}}$.
\end{proof}

\subsection{Ekedahl-Oort stratifications of lattice strata} The goal of this section is to define, for each vertex lattice $L_{i}$ of type $i$ with $i=0, 2$, a set theoretic \emph{Ekedahl-Oort stratification} of the lattice stratum $\calM_{L_{i}}(\FF)$. We define the following subsets of $\calM_{L_{0}}(\FF)$ which we will refer to as the \emph{Ekedahl-Oort strata}
\begin{equation}
\begin{split}
&\mathcal{M}^{\circ}_{L_{0}}(\FF) =\mathcal{M}_{L_{0}}(\FF) -( \mathcal{M}_{\{2\}}(\FF)\cup \mathcal{M}_{\{1\}}(\FF));\\
&\calM_{L_{0},\{2\}}(\FF) =\calM_{L_{0}}(\FF)\cap \calM_{\{2\}}(\FF);\\
&\calM^{\circ}_{L_{0},\{2\}}(\FF) =(\calM_{L_{0}}(\FF)\cap \calM_{\{2\}}(\FF))- \calM_{\{1\}}(\FF);\\
&\calM_{L_{0},\{1\}}(\FF)=\calM_{L_{0}}(\FF)\cap \calM_{\{1\}}(\FF).\\
\end{split}
\end{equation}
Then $\calM_{L_{0}}(\FF)$ admits a decomposition which we call the \emph{Ekedahl-Oort stratification}  
\begin{equation}
\calM_{L_{0}}(\FF)=\calM^{\circ}_{L_{0}}(\FF)\sqcup \calM^{\circ}_{L_{0},\{2\}}(\FF)\sqcup \calM_{L_{0}, \{1\}}(\FF).
\end{equation}
In the proposition below we will describe the Ekedhal-Oort strata in terms of Deligne-Lusztig varieties. We define the symplectic space $V(L_{0})$ of dimension $4$ over $\FF$  by $L_{0, W_{0}}/pL^{\vee}_{0, W_{0}}$ and  we put $Y_{L_{0}}=Y^{(-)}_{V(L_{0})}$.

\begin{proposition}\label{DL0}
There is a bijection between $\mathcal{M}_{L_{0}}(\FF)$ and $Y_{L_{0}}(\FF)$. This bijection is compatible with the stratification on both sides in the sense that
\begin{itemize} 
\item[-] the stratum $\calM^{\circ}_{L_{0}}(\FF)$ is in bijection with $X_{B}(w_{2})(\FF)$;  
\item[-] the stratum $\calM^{\circ}_{L_{0},\{2\}}(\FF)$ is in bijection with $X_{B}(w_{1})(\FF)$;
\item[-] the stratum $\calM_{L_{0}, \{1\}}(\FF)$ is in bijection with $X_{P_{\{1\}}}(1)(\FF)$.
\end{itemize}
\end{proposition}

\begin{proof}
Given a point $D\in \mathcal{M}_{L_{0}}(\FF)$, we have the following chain of inclusions:
$$pL_{0,W_{0}}^{\vee}\subset^{1}pD^{\vee}\subset^{2}D\subset^{1}L_{0,W_{0}}.$$
Recall that $V(L_{0})=L_{0,W_{0}}/pL_{0,W_{0}}^{\vee}$, then sending $D$ to the flag 
$$0\subset^{1}U^{\perp}=pD^{\vee}/pL_{0,W_{0}}^{\vee} \subset^{2}U=D/pL_{0,W_{0}}^{\vee}\subset^{1}L_{0,W_{0}}/pL_{0,W_{0}}^{\vee}$$ gives a map from $\mathcal{M}_{L_{0}}(\FF)$ to $Y_{L_{0}}(\FF)$. 

Given $D\in \mathcal{M}^{\circ}_{L_{0}}(\FF) $, $D\cap \tau(D)$ is not $\tau$-stable by definition of  $\mathcal{M}^{\circ}_{L_{0}}(\FF)$ and hence $U\cap \sigma(U)$ is not $\sigma$-stable. This gives a point in $X_{B}(w_{2})(\FF)$.

Given $D\in \mathcal{M}^{\circ}_{L_{0}, \{2\}}(\FF) $, then there is a vertex lattice $L_{2}$ of type $2$ such that $$pL^{\vee}_{0,W_{0}}\subset^{1} pD^{\vee}\subset^{1} L_{2, W_{0}} \subset^{1} D\subset^{1} L_{0, W_{0}}.$$ Hence comparing $D\cap \tau(D)$ and $L_{2, W_{0}}$ shows that $L_{2, W_{0}}=D\cap \sigma(D)$. Then $U^{\perp}= pD^{\vee}/pL^{\vee}_{0, W_{0}}$ is not $\sigma$-stable but $U\cap \sigma(U)$ is $\sigma$-stable and this is precisely in $X_{B}(w_{1})(\FF)$.

Given $D\in \mathcal{M}^{\circ}_{L_{0}, \{1\}}(\FF) $, then there is vertex lattice $L_{1}$ of type $1$ such that  $$pL^{\vee}_{0, W_{0}}\subset^{1} pD^{\vee}\subset pL^{\vee}_{1, W_{0}}\subset^{2} L_{1,W_{0}} \subset D\subset^{1} L_{0, W_{0}}.$$ This forces $D=L_{1}$ and $U^{\perp}=pD^{\vee}/pL^{\vee}_{0, W_{0}}$ is $\sigma$-stable and this is precisely in $X_{P_{\{1\}}}(1)(\FF)$.

Conversely, if $0\subset U^{\perp}\subset U\subset V(L_{0})=L_{0, W_{0}}/pL^{\vee}_{0, W_){0}}$, then define $D$ as the preimage of $U$ in $L_{0, W_{0}}$. One check easily this gives the desired inverse map.
\end{proof}

We define the following subsets of $\calM_{L_{2}}(\FF)$ similarly as in the previous case
\begin{equation}
\begin{split}
&\mathcal{M}^{\circ}_{L_{2}}(\FF)=\mathcal{M}_{L_{2}}(\FF) -( \mathcal{M}_{\{0\}}(\FF)\cup \mathcal{M}_{\{1\}}(\FF));\\ 
&\calM_{L_{2},\{0\}}(\FF)=\calM_{L_{2}}(\FF)\cap \calM_{\{0\}}(\FF);\\
&\calM^{\circ}_{L_{2},\{0\}}(\FF)=(\calM_{L_{2}}(\FF)\cap \calM_{\{0\}}(\FF))- \calM_{\{1\}}(\FF);\\ 
&\calM_{L_{2},\{1\}}(\FF)=\calM_{L_{2}}(\FF)\cap \calM_{\{1\}}(\FF).\\ 
\end{split}
\end{equation}
Then $\calM_{L_{2}}(\FF)$ admits also the \emph{Ekedhal-Oort stratification} 
\begin{equation}
\calM_{L_{2}}(\FF)=\calM^{\circ}_{L_{2}}(\FF)\sqcup \calM^{\circ}_{L_{2}, \{0\}}(\FF)\sqcup\calM_{L_{2},\{1\}}(\FF).
\end{equation}
We define the symplectic space $V(L_{2})$ of dimension $4$ over $\FF$ by $L^{\vee}_{2, W_{0}}/L_{2, W_{0}}$ and we put $Y_{L_{2}}=Y^{(+)}_{V(L_{2})}$. 

\begin{proposition}\label{DL2}
There is a bijection between $\mathcal{M}_{L_{2}}(\FF)$ and $Y_{L_{2}}(\FF)$. This bijection is compatible with the stratification on both sides in the sense that
\begin{itemize}  
\item[-] The stratum $\calM^{\circ}_{L_{2}}(\FF)$ is in bijection with $X_{B}(w_{2})(\FF)$; 
\item[-] The stratum $\calM^{\circ}_{L_{2},\{0\}}(\FF)$ is in bijection with $X_{B}(w_{1})(\FF)$;
\item[-] The stratum $\calM_{L_{2}, \{1\}}(\FF)$ is in bijection with $X_{P_{\{1\}}}(1)(\FF)$.
\end{itemize}
\end{proposition}

\begin{proof}
Given $D\in \mathcal{M}_{L_{2}}(\FF)$, we have the following chain of inclusions:
$$L_{2,W_{0}} \subset^{1}D\subset^{2}D^{\vee}\subset^{1}L^{\vee}_{2,W_{0}}.$$
Let $V(L_{2})= L^{\vee}_{2, W_{0}}/ L_{2, W_{0}}$, then sending $D$ to the flag $$0\subset^{1}D/L_{2, W_{0}} \subset^{2}D^{\vee}/L_{2, W_{0}} \subset^{1}L^{\vee}_{2, W_{0}}/L_{2, W_{0}} $$ gives the desired map from $\mathcal{M}_{L_{2}}(\FF)$ to $Y_{L_{2}}(\FF)$. The rest of proof proceeds similarly as in the previous case.
\end{proof}

For a vertex lattice $L_{1}$ of type $1$, recall that the stratum $\mathcal{M}_{L_{1}}(\FF)$ is the set $$ \{D \in \mathcal{M}(\FF): L_{1,W_{0}}\subset D\}.$$ This stratum consists of a superspecial point.
\begin{proposition}\label{L1}
The stratum $\mathcal{M}_{L_{1}}(\FF)$ consists of a single superspecial point and $\calM_{\{1\}}$ consists of all the superspecial points.
\end{proposition}
\begin{proof}
Given $D\in \mathcal{M}_{L_{1}}(\FF)$, we have the following chain:
$$L_{1,W_{0}} \subset D\subset D^{\vee}\subset L^{\vee}_{1,W_{0}}.$$
This forces $L_{1,W_{0}} = D$. Since $\tau(L_{1, W_{0}})=L_{1, W_{0}}$, it follows that $VD=\Pi D$. As $\Pi^{2}=p$, $D$ is superspecial. Note that $M$ is superspecial if and only if $M_{0}=D$ is superspecial. The result follows. It is also clear that all the superspecial points arise this way.
\end{proof}

\begin{lemma}\label{L0_L2_intersection}
Let $L_{1}$ be a vertex lattice of type $1$. Then we can find a vertex lattice $L_{0}$ of type $0$ containing $L_{1}$ and a vertex lattice $L_{2}$ of type $2$ contained in $L_{1}$. 
\end{lemma}

\begin{proof}
Since $L_{1}$ is a vertex lattice of type $1$, $pL_{1}^{\vee}\subset^{2}L_{1}\subset^{2} L_{1}^{\vee}$. Consider the symplectic space $L_{1}^{\vee}/L_{1}$ of dimension two. Choose any line $\bar{l}$ in $L^{\vee}_{1}/L_{1}$ and denote by $L_{0}$ its lift in $L^{\vee}_{1}$. One verifies immediately that $L_{0}$ is a vertex lattice type $0$. Similarly choose any line in $L_{1}/pL^{\vee}_{1}$ gives a vertex lattice $L_{2}$ of type $2$. 
\end{proof}

This lemma shows that any superspecial point is contained in the intersection $$\calM_{L_{0}}(\FF)\cap \calM_{L_{2}}(\FF)$$ for a vertex lattice $L_{0}$ of type $0$  and a vertex lattice $L_{2}$ of type $2$.  Finally we study the intersection of any two $2$-dimensional lattice strata. We have already seen that for a vertex lattice $L_{0}$ of type $0$ and a vertex lattice $L_{2}$ of type $2$, the intersection $\calM_{L_{0}}(\FF)\cap \calM_{L_{2}}(\FF)$ is $\PP^{1}(\FF)$. Let $L_{0}$ and $L^{\prime}_{0}$ be two vertex lattices of type $0$.  Similarly let $L_{2}$ and $L^{\prime}_{2}$ be two vertex lattices of type $2$.

\begin{lemma}
Suppose the intersection of $\calM_{L_{0}}(\FF)$ and $\calM_{L^{\prime}_{0}}(\FF)$ is nonempty, then it consists of a superspecial point corresponding to the vertex lattice $L_{0}\cap L^{\prime}_{0}$ of type $1$. Similarly if the intersection of $\calM_{L_{2}}(\FF)$ and $\calM_{L^{\prime}_{2}}(\FF)$ is nonempty, then it consists of a superspecial point corresponding to the vertex lattice $L_{2}+ L^{\prime}_{2}$ of type $1$.
\end{lemma}
\begin{proof}
Suppose $M\in \calM_{L_{0}}\cap \calM_{L^{\prime}_{0}}$. Then $pL^{\vee}_{0,W_{0}}\subset^{1} pM^{\vee}\subset^{2} M\subset^{1} L_{0, W_{0}}$ and $p(L^{\prime}_{0, W_{0}})^{\vee}\subset^{1} pM^{\vee}\subset^{2} M\subset^{1} L^{\prime}_{0, W_{0}}$, this forces $M=L_{0,W_{0}}\cap L^{\prime}_{0,W_{0}}$. Moreover $L_{0}\cap L^{\prime}_{0}$ is a vertex lattice of type $1$. The statement for $L_{2}, L^{\prime}_{2}$ is proved in a similar way.
\end{proof}

\section{The isogeny trick}\label{Isogeny-trick} For each vertex lattice $L$, we now would like to equip the set $\calM_{L}(\FF)$ with a scheme structure. If $L_{0}$ is a vertex lattice of type $0$, we define $L^{+}_{0}= L_{0,W_{0}}\oplus \Pi^{-1} L_{0, W_{0}}$ and $L_{0}^{-}=(L_{0}^{+})^{\perp}$ where $\perp$ is the integral dual taken with respect to the form $(\cdot,\cdot)$ inside of $N$.  Notice by an easy computation $L_{0,W_{0}}^{\perp}=\Pi L_{0, W_{0}}^{\vee}=\Pi L_{0, W_{0}}$ and $(\Pi^{-1} L_{0,W_{0}})^{\perp}= pL^{\vee}_{0,W_{0}}=pL_{0,W_{0}}$.

\begin{lemma}
The lattices $L_{0}^{+}$ and $L_{0}^{-}$ give rise to  $p$-divisible groups $\XX_{L_{0}^{+}}$ and $\XX_{L_{0}^{-}}$ equipped with two quasi-isogenies $\rho_{L^{+}_{0}}:\XX_{L^{+}_{0}}\rightarrow \XX$ and $\rho_{L^{-}_{0}}:\XX_{L^{-}_{0}}\rightarrow \XX$. Moreover there is a commutative diagram:
$$
\begin{tikzcd}
 \XX_{L_{0}^{+}} \arrow[r,  "\sim"] \arrow[d,  "\rho_{L_{0}^{+}}"]
    & \XX^{\vee}_{L_{0}^{-}} \\
   \XX \arrow[r,  "\lambda_{\XX}"] & \XX^{\vee}  \arrow[u,  "\rho^{\vee}_{L_{0}^{-}}"] .
\end{tikzcd}$$ 
\end{lemma}

\begin{proof}
To see $L_{0}^{+}$ defines a $p$-divisible groups, we need to verify that $p L_{0}^{+} \subset V L_{0}^{+} \subset L_{0}^{+}$. This is equivalent to $p\Pi^{-1}L_{0,W_{0}}\subset V L_{0,W_{0}}\subset \Pi^{-1} L_{0,W_{0}}$ and $pL_{0,W_{0}}\subset V\Pi^{-1} L_{0, W_{0}}\subset  L_{0, W_{0}}$. Notice that $\Pi=V=F$ on $L_{0, W_{0}}$ and the two conditions above are easily verified. Therefore $L^{+}_{0}$ gives a $p$-divisible group $\XX_{L^{+}_{0}}$. The dual $L^{-}_{0}$ of $L^{+}_{0}$ gives the dual $p$-divisible group $\XX_{L^{-}_{0}}\cong \XX^{\vee}_{L^{+}_{0}}$. The maps $\rho_{L^{+}_{0}}$ and $\rho_{L^{-}_{0}}$ are given by the inclusions of $L^{+}_{0}\subset N$ and $L^{-}_{0}\subset N$.
\end{proof}

Let  $(X, \iota_{X},  \lambda_{X}, \rho_{X})\in \calM(R)$  for a $\FF$-algebra $R$, consider the quasi-isogenies defined by 
\begin{equation}\label{rho0+}
\rho_{X,L_{0}^{+}} :X \xrightarrow{\rho_{X}} \XX_{R} \xrightarrow{\rho^{-1}_{L_{0}^{+}}} \XX_{L_{0}^{+}, R} 
 \end{equation}
and
\begin{equation}\label{rho0-}
\rho_{X,L_{0}^{-}} :\XX_{L_{0}^{-}, R} \xrightarrow{\rho_{L_{0}^{-}}} \XX_{R} \xrightarrow{\rho_{X}^{-1}} X.
 \end{equation}

 \begin{definition} We define lattice stratum $\calM_{L_{0}}$ associated to $L_{0}$ as the subscheme of $\calM$ consisting of those points that $\rho_{X,L_{0}^{+}}$ is an isogeny. This is equivalent to $\rho_{X,L_{0}^{-}}$ is an isogeny.
\end{definition}

\begin{proposition}\label{M0proj}
The subscheme $\calM_{L_{0}}$ is a closed subscheme of $\calM$. Moreover $\calM_{L_{0}}$ is projective.
\end{proposition}

\begin{proof}
First of all, $\calM_{L_{0}}$ is a closed subscheme by \cite[Proposition 2.9]{RZ-Aoms}. Moreover $\calM_{L_{0}}$ is bounded in the sense of \cite[2.30]{RZ-Aoms}. This follows from the fact that we have constructed the two isogenies
$$\XX_{L_{0}^{-}, R} \xrightarrow{\rho_{X,L_{0}^{-}}} X  \xrightarrow{\rho_{X,L_{0}^{+}}}\XX_{L_{0}^{+}, R}.$$
Thus $\calM_{L_{0}}$ is closed subscheme of a projective scheme by \cite[Corollary 2.29]{RZ-Aoms}. Thus it is projective itself.
\end{proof}

\begin{lemma}\label{M0set}
The set $\calM_{L_{0}}(\FF)$ in Definition \ref{stratum-set} is precisely the set of $\FF$-points of $\calM_{L_{0}}$.
\end{lemma}
\begin{proof}
Giving a $\FF$-point of $\calM_{L_{0}}$ is equivalent to giving a Dieudonn\'{e} module $M\subset L^{+}_{0}$ but this is equivalent to $M_{0}\subset L_{0,W_{0}}$. Indeed, if $M\subset L^{+}_{0}$, then $M_{0}\subset L_{0,W_{0}}$ is obvious. Conversely if $M_{0}\subset L_{0, W_{0}}$, then $M_{1}= \Pi M^{\vee}_{0}\subset \Pi^{-1}M_{0}\subset \Pi^{-1}L_{0,W_{0}}$.
\end{proof}

If $L_{2,W_{0}}$ is a vertex lattice of type $2$, we define $L_{2}^{+}= L_{2, W_{0}}\oplus \Pi L_{2, W_{0}}$ and $L_{2}^{-}=(L_{2}^{+})^{\perp}$ where $\perp$ is taken with respect to the form $(\cdot,\cdot)$ inside of $N$.  The proof of the following statements are proved exactly the same way as in the $L_{0}$ case.

\begin{lemma}
The lattices $L_{2}^{+}$ and $L_{2}^{-}$ gives rise to  $p$-divisible groups $\XX_{L_{2}^{+}}$ and $\XX_{L_{2}^{-}}$ equipped with quasi-isogenies $\rho_{L^{+}_{2}}:\XX_{L^{+}_{2}}\rightarrow \XX$ and $\rho_{L^{-}_{2}}:\XX_{L^{-}_{2}}\rightarrow \XX$.
Moreover there is a commutative diagram:
$$
\begin{tikzcd}
 \XX_{L_{2}^{+}} \arrow[r,  "\sim"] \arrow[d,  "\rho_{L_{2}^{+}}"]
    & \XX^{\vee}_{L_{2}^{-}} \\
   \XX \arrow[r,  "\lambda_{\XX}"] & \XX^{\vee}  \arrow[u,  "\rho^{\vee}_{L_{2}^{-}}"] .
\end{tikzcd}$$ 
\end{lemma}

Let  $(X, \iota_{X},  \lambda_{X}, \rho_{X})$  for a $\FF$-algebra $R$, consider the quasi-isogeny defined by 
\begin{equation}\label{rho2+}
\rho_{X,L_{2}^{+}} : \XX_{L_{2}^{+}, R}  \xrightarrow{\rho_{L_{2}^{+}}}\XX_{R} \xrightarrow{\rho^{-1}_{X}} X,
 \end{equation}
 and
 \begin{equation}\label{rho2-}
\rho_{X,L_{2}^{-}} : X \xrightarrow{\rho_{X}} \XX_{R} \xrightarrow{\rho^{-1}_{L^{-}_{2}}} \XX_{L_{2}^{-}, R} .
 \end{equation}
\begin{definition}
We define the lattice stratum $\calM_{L_{2}}$ associated to $L_{2}$ as the subscheme of $\calM$ consisting of those points that $\rho_{X,L_{2}^{+}}$ is an isogeny or equivalently $\rho_{X,L_{2}^{-}}$ is an isogeny.
\end{definition}
\begin{proposition}
The subscheme $\calM_{L_{2}}$ is a closed subscheme of $\calM$. Moreover $\calM_{L_{2}}$ is projective. The set $\calM_{L_{2}}(\FF)$ in Definition \ref{stratum-set} is precisely set of $\FF$-points of $\calM_{L_{2}}$.
\end{proposition}
\begin{proof}
This is proved in the exactly same way as in Proposition \ref{M0proj} and Lemma \ref{M0set}.
\end{proof}

\subsection{The isomorphism between $\calM_{L}$ and $Y_{L}$} Let $L_{0}$ be a vertex lattice of type $0$ and $L_{2}$ a vertex lattice of type $2$. The goal of this subsection is to define a map $f_{i}: \calM_{i}\rightarrow Y_{L_{i}} $ for $i=0,2$ and show that it is an isomorphism.

We begin with $L_{0}$. Given a point $(X, \iota_{X},  \lambda_{X}, \rho_{X})\in \calM_{L_{0}}(R)$ for a $\FF$-algebra $R$. Consider the isogenies $$\rho_{L_{0}}:\XX_{L_{0}^{-}, R} \xrightarrow{\rho_{X,L_{0}^{-}}} X  \xrightarrow{\rho_{X,L_{0}^{+}}}\XX_{L_{0}^{+}, R}.$$
The kernel of $\rho_{L_{0}}$ is the quotient $L^{+}_{0, R}/L^{-}_{0, R}$ which we compute to be $L_{0,R}/pL^{\vee}_{0,R}\oplus \Pi^{-1}L_{0,R}/\Pi L^{\vee}_{0,R}$. The kernel of $\rho_{X,L^{-}_{0}}$ is a direct summand of $L_{0, R}/pL^{\vee}_{0, R}\oplus \Pi^{-1}L_{0, R}/\Pi L^{\vee}_{0, R}$ which we compute to be $M_{0}/pL^{\vee}_{0, R}\oplus\Pi M^{\vee}_{0}/\Pi L^{\vee}_{0, R}$ where $M=M_{0}\oplus M_{1}$ is the covariant Dieudonn\'{e} crystal $\DD(X)(R)$ of $X$ \cite{BBM82}. Sending $(X, \iota_{X},  \lambda_{X}, \rho_{X})\in \calM_{L_{0}}(R)$ to $M_{0}/ pL^{\vee}_{0, R}\subset L_{0,R}/pL^{\vee}_{0,R}$ gives a map from $\calM_{L_{0}}$ to $Y_{L_{0}}$. We denote this map by $f_{0}$.

We consider a vertex lattice $L_{2}$ of type $2$. Given a point $(X, \iota_{X},  \lambda_{X}, \rho_{X})\in \calM_{L_{2}}(R)$ for a $\FF$-algebra $R$. Consider the isogenies $$\rho_{L_{2}}:\XX_{L_{2}^{+}, R} \xrightarrow{\rho_{X,L_{2}^{+}}} X  \xrightarrow{\rho_{X,L_{2}^{-}}}\XX_{L_{2}^{-}, R}.$$
The kernel of $\rho_{L_{0}}$ is the quotient $L^{-}_{2, R}/L^{+}_{2, R}$ which we compute to be $L^{\vee}_{2,R}/L_{2,R}\oplus \Pi L^{\vee}_{2,R}/\Pi L_{2,R}$. The kernel of $\rho_{X,L^{-}_{0}}$ is a direct summand of $L^{\vee}_{2,R}/L_{2,R}\oplus \Pi L^{\vee}_{2,R}/\Pi L_{2,R}$ which we compute to be $M_{0}/L_{2, R}\oplus\Pi M^{\vee}_{0}/\Pi L_{2, R}$ where $M=M_{0}\oplus M_{1}$ is the covariant Dieudonn\'{e} crystal $\DD(X)(R)$ of $X$. Sending $(X, \iota_{X},  \lambda_{X}, \rho_{X})\in \calM_{L_{2}}(R)$ to $M_{0}/ L_{2, R}$ gives a map from $\calM_{L_{2}}$ to $Y_{L_{2}}$. We denote this map by $f_{2}$.

\begin{proposition}\label{isom0}
The maps $f_{i}: \calM_{L_{i}}\rightarrow Y_{L_{i}}$  are isomorphisms for $i=0,2$. 
\end{proposition}

\begin{proof}
Using the theory of windows of display \cite{Zink-pro99} in place of Dieudonn\'{e} modules, one can show $f_{0}$ is bijective on $k$-points for any field extension $k$ over $\FF$ following the same proof as in Proposition \ref{DL0}.  Since $Y_{L_{0}}$ is a closed subscheme of a projective scheme, it is projective. Moreover it is smooth and hence normal. Note that $f_{0}$ is proper since $\calM_{L_{0}}$ is proper and $Y_{L_{0}}$ is separated. Therefore we can apply Zariski's main theorem, $f_{0}$ is an isomorphism. The proof for $f_{2}$ is exactly the same.
\end{proof}

\subsection{The Bruhat-Tits stratification in the quaternionic unitary case} We can now transfer the results on the Bruhat-Tits stratification from the set theoretic level in Proposition \ref{DL0} and Proposition \ref{DL2} to the scheme theoretic level. 

For $i=0,1,2$, we define the \emph{closed Bruhat-Tits stratum of type $i$} to be $$\mathcal{M}_{\{i\}}=\bigcup_{L_{i}}\calM_{L_{i}}$$ where $L_{i}$ ranges over all the vertex lattices of type $i$. 
\begin{definition}\label{EO-strata}
We define the Ekedhal-Oort strata of $\calM_{L_{0}}$ by the following schemes.
\begin{enumerate}
 \item $\mathcal{M}^{\circ}_{L_{0}}=\mathcal{M}_{L_{0}} -( \mathcal{M}_{\{2\}}\cup \mathcal{M}_{\{1\}})$, this is an open stratum;
 \item $\calM_{L_{0},\{2\}}=\calM_{L_{0}}\cap \calM_{\{2\}}$, this is a closed stratum;
 \item $\calM^{\circ}_{L_{0},\{2\}}=(\calM_{L_{0}}\cap \calM_{\{2\}})- \calM_{\{1\}}$;
 \item $\calM_{L_{0},\{1\}}=\calM_{L_{0}}\cap \calM_{\{1\}}$, this is a closed stratum. 
 \end{enumerate}
We  call the natural decomposition 
\begin{equation*}
\calM_{L_{0}}= \mathcal{M}^{\circ}_{L_{0}}\sqcup \calM^{\circ}_{L_{0},\{2\}} \sqcup \calM_{L_{0},\{1\}} 
\end{equation*}
the Ekedahl-Oort stratification of $\calM_{L_{0}}$.

Similarly, we define the Ekedahl-Oort strata of $\calM_{L_{2}}$ by the following schemes.
\begin{enumerate}
 \item $\mathcal{M}^{\circ}_{L_{2}}=\mathcal{M}_{L_{2}} -( \mathcal{M}_{\{0\}}\cup \mathcal{M}_{\{1\}})$, this is an open stratum;
 \item $\calM_{L_{2},\{0\}}=\calM_{L_{2}}\cap \calM_{\{0\}}$, this is a closed stratum;
 \item $\calM^{\circ}_{L_{2},\{0\}}=(\calM_{L_{2}}\cap \calM_{\{0\}})- \calM_{\{1\}}$;
 \item $\calM_{L_{2},\{1\}}=\calM_{L_{2}}\cap \calM_{\{1\}}$, this is a closed stratum. 
 \end{enumerate}
We  call the natural decomposition 
\begin{equation*}
\calM_{L_{2}}= \mathcal{M}^{\circ}_{L_{2}}\sqcup \calM^{\circ}_{L_{2},\{0\}} \sqcup \calM_{L_{2},\{1\}}
\end{equation*}
the Ekedahl-Oort stratification of $\calM_{L_{2}}$.
\end{definition}

\begin{theorem}\label{stratification}
Let $L_{i}$ be a vertex lattice of type $i$ for $i=0, 2$. The isomorphism $$f_{i}: \calM_{L_{i}}\rightarrow Y_{L_{i}}$$ is compatible with the stratification on both sides: on $\calM_{L_{i}}$ we consider the Ekedahl-Oort stratification and on $Y_{L_{i}}$ we consider the stratification in Theorem \ref{DL-stratum}.
\end{theorem}

\begin{proof}
Once we know that $f$ is an isomorphism, the statements about the stratification can be checked on $\FF$-points. These are already done in Proposition \ref{DL0} and Proposition \ref{DL2}. More explicitly, for $L_{0}$ we see that
\begin{enumerate}
\item $\calM^{\circ}_{L_{0}}$ is isomorphic to $X_{B}(w_{2})$;
\item $\calM^{\circ}_{L_{0},\{2\}}$ is isomorphic to $X_{B}(w_{1})$;
\item  $\calM_{L_{0},\{1\}}$ is isomorphic to $X_{P_{\{2\}}}(1)$.
\end{enumerate}
And for $L_{2}$, we have
\begin{enumerate}
\item $\calM^{\circ}_{L_{2}}$ is isomorphic to $X_{B}(w_{2})$;
\item $\calM^{\circ}_{L_{2},\{0\}}$ is isomorphic to $X_{B}(w_{1})$;
\item  $\calM_{L_{2},\{1\}}$ is isomorphic to $X_{P_{\{2\}}}(1)$.
\end{enumerate}
\end{proof}

We finally define the Bruhat-Tits stratification of $\calM$ by 
\begin{equation}
\begin{split}
&\calM^{\circ}_{\{0\}}=\bigcup_{L_{0}}\calM^{\circ}_{L_{0}}\text{ where $L_{0}$ runs through all vertex lattices of type $0$};\\
&\calM^{\circ}_{\{2\}}=\bigcup_{L_{2}}\calM^{\circ}_{L_{2}}\text{ where $L_{2}$ runs through all vertex lattices of type $2$};\\ 
&\calM_{\{1\}}=\bigcup_{L_{1}}\calM_{L_{1}}\text{ where $L_{1}$ runs through all vertex lattices of type $1$} ;\\
&\calM^{\circ}_{\{02\}}=\bigcup_{L_{0}}\calM^{\circ}_{L_{0}, \{2\}}\text{ where $L_{0}$ runs through all vertex lattices of type $0$}. \\
\end{split}
\end{equation}

We refer to the following natural decomposition of $\calM$ the \emph{Bruhat-Tits stratification} of $\calM$
\begin{equation}
\calM=\calM^{\circ }_{\{0\}}\sqcup\calM^{\circ}_{\{2\}}\sqcup \calM^{\circ}_{\{0,2\}}\sqcup \calM_{\{1\}}.
\end{equation}

\section{The supersingular locus of the quaternionic unitary Shimura variety}

\subsection{The main result on the Rapoport-Zink space}\label{main-result-section} We summarize the results from previous sections and finish the description of the Rapoport-Zink space in the following theorem. 

\begin{theorem}\label{main-result}
The formal scheme $\calN$ can be written as $\calN=\bigsqcup_{i\in \ZZ} \calN(i)$. The connected components $\calN(i)$ are all isomorphic to $\calN(0)$. Let $\calM=\calN_{red}(0)$ be the underlying reduced scheme of $\calN(0)$, then $\calM$ is pure of dimension $2$.
\begin{enumerate}
\item $\calM$ admitts the Bruhat-Tits stratification $$\calM=\calM^{\circ }_{\{0\}}\sqcup\calM^{\circ}_{\{2\}}\sqcup \calM^{\circ}_{\{0,2\}}\sqcup \calM_{\{1\}}.$$ Here $\calM^{\circ}_{\{0\}}=\bigcup_{L_{0}}\calM^{\circ}_{L_{0}}$ where $L_{0}$ runs through all the vertices of type $0$. For  each $L_{0}$, the closure of $\calM^{\circ}_{L_{0}}$ is $\calM_{L_{0}}$ which  is isomorphic to the hypersurface $x_{3}^{p}x_{0}-x^{p}_{0}x_{3}+x^{p}_{2}x_{1}-x^{p}_{1}x_{2}=0$.  The scheme $\calM_{L_{0}}$ admits a stratification $$\calM_{L_{0}}=\calM^{\circ}_{L_{0}}\sqcup \calM^{\circ}_{L_{0},\{2\}}\sqcup \calM_{L_{0}, \{1\}}$$ called the Ekedahl-Oort stratification of $\calM_{L_{0}}$. Here the closure of $\calM^{\circ}_{L_{0},\{2\}}$ is $\calM_{L_{0},\{2\}}$ and its irreducible components are isomorphic to $\PP^{1}$. The complement of $\calM^{\circ}_{L_{0},\{2\}}$ in $\calM_{L_{0},\{2\}}$ is precisely $\calM_{L_{0}, \{1\}}$ which consists of superspecial points. 

The scheme $\calM^{\circ}_{\{2\}}=\bigcup_{L_{2}}\calM^{\circ}_{L_{2}}$ where $L_{2}$ runs through all the vertices of type $2$. For  each $L_{2}$, the closure of $\calM^{\circ}_{L_{2}}$ is $\calM_{L_{2}}$ which  is isomorphic to the hypersurface $x_{3}^{p}x_{0}-x^{p}_{0}x_{3}+x^{p}_{2}x_{1}-x^{p}_{1}x_{2}=0$. Moreover it admits a stratification $$\calM_{L_{2}}=\calM^{\circ}_{L_{2}}\sqcup \calM^{\circ}_{L_{2},\{0\}}\sqcup \calM_{L_{2}, \{1\}}$$ called the Ekedahl-Oort stratification of $\calM_{L_{2}}$. Here the closure of $\calM^{\circ}_{L_{2},\{0\}}$ is $\calM_{L_{2},\{0\}}$ and its irreducible components are isomorphic to $\PP^{1}$. The complement of $\calM^{\circ}_{L_{2},\{0\}}$ in $\calM_{L_{2},\{0\}}$ is precisely $\calM_{L_{2}, \{1\}}$ which consists of superspecial points.

\item The intersection between $\calM_{L_{0}}$ and $\calM_{L_{2}}$ for a vertex lattice $L_{0}$ of type $0$ and a vertex lattice $L_{2}$ of type $2$, if nonempty, is isomorphic to a $\PP^{1}$. The intersection between $\calM_{L_{0}}$ and $\calM_{L^{\prime}_{0}}$ for a vertex lattice $L_{0}$ and a different vertex lattice $L^{\prime}_{0}$ of type $0$, if nonempty, is a point which is superspecial. The intersection between $\calM_{L_{2}}$ and $\calM_{L^{\prime}_{2}}$ for a vertex lattice $L_{2}$ of type $2$ and another vertex lattice $L^{\prime}_{2}$ of type $2$, if nonempty, is a superspecial pint. 
\end{enumerate}
\end{theorem}

\subsection{Integral model of the Shimura variety} Let $V=B\oplus B$ considered as a vector space of dimension $8$ over $\QQ$. We assume that $V$ is equipped with an alternating form $(\cdot,\cdot)$ such that $(x, by)=(b^{*}x, y)$. Then we define $G(\QQ)=\{g\in \GL_{B}(V); (gx, gy)=c(g)(x, y)\}$. Since $B$ is split at $\infty$, $G(\RR)=\GSp(4)(\RR)$. Let $h: \GG_{m}\rightarrow G(\CC)$ be the cocharacter sending $z$ to $\text{diag}(z,z,1,1)$. Moreover $h$ defines a decomposition $V_{\CC}=V_{1}\oplus V_{2}$ where $h(z)$ acts on $V_{1}$ by $z$ and on $V_{2}$ by $\bar{z}$. We fix an open compact subgroup $U=U^{p}U_{p}$ of $G(\AAA_{f})$ and we assume $U^{p}$ is sufficiently small and $U_{p}=G(\ZZ_{p})$. We also assume that there is an $\calO_{B}$ lattice $\Lambda$ in $V$ such that $G(\ZZ_{p})$ stabilizes $\Lambda\otimes \ZZ_{p}$.

To $(B,*,  V, (\cdot,\cdot), h ,\Lambda)$, we associate the following moduli problem $\gothS h_{U}$ over $\ZZ_{(p)}$. To a scheme $S$ over $\ZZ_{(p)}$, we classify the set of isomorphism classes of the quadruple $(A, \iota , \lambda, \eta )$ where:
\begin{itemize}
\item[-]  $A$ is an abelian scheme of relative dimension $4$ over $S$;
\item[-]  $\lambda: A\rightarrow A^{\vee}$ is a prime to $p$ polarization which is $\calO_{B}$-linear;
\item[-]  $\iota: \calO_{B} \rightarrow \End_{S}(A)\otimes \ZZ_{(p)}$ is a morphism such that $\lambda\circ i(a^{*})= i(a)^{\vee}\circ\lambda$ and satisfies the Kottwitz condition \begin{equation}\label{Kowtt_Shim}\det(T-\iota(a); \Lie(A))=(T^{2}-\mathrm{Trd}^{0}(a)T+\mathrm{Nrd}^{0}(a))^{2}\end{equation} for all $a\in \calO_{B}$ with $\mathrm{Trd}^{0}$ the reduced trace on $B$ and $\mathrm{Nrd}^{0}$ the reduced norm on $B$;
\item[-] $\eta: V\otimes_{\QQ} {\AAA}^{(p)}_{f} \rightarrow V^{(p)}(A)$ is a $U^{p}$-orbit of $B\otimes_{\QQ}\AAA^{(p)}_{f}$-linear isomorphisms where $$\prod_{p^{\prime}\neq p}\Ta_{p^{\prime}}(A)\otimes \AAA_{f}^{(p)}.$$ We require that $\eta$ is compatible with the Weil-pairing on the righthand side and the form $(\cdot,\cdot)$ on the lefthand side.
\end{itemize}

This moduli problem is representable by a quasi-projective variety $\Sh_{U^{p}}$  over $\ZZ_{(p)}$ cf. \cite[ Proposition 1.3]{KR-ASENS94}. Note that in \cite{KR-ASENS94}, the moduli space is presented in terms of a $\GSpin$ Shimura variety but the equivalence between the two moduli problem is clear.

We are interested in the closed subscheme  of  the supersingular locus $Sh^{ss}_{U^{p}}$ in $\Sh_{U^{p}}\otimes \FF$ considered as a closed reduced subscheme. The uniformization theorem of Rapoport-Zink \cite{RZ-Aoms} Theorem 6.1 translates this problem to problem of  describing the corresponding Rapoport-Zink space.

\begin{theorem}\label{uniformization}
There is an isomorphism of $\FF$-schemes
$$Sh^{ss}_{U^{p}}\cong I(\QQ)\backslash \calN_{red}\times  G(\AAA^{p}_{f})/U^{p}.$$
\end{theorem} 

Here $I$ is an inner form of $G$ with $I(\QQ_{p})=J_{b}(\QQ_{p})$. It is defined in the following way. Let  $(\bf{A}, \boldsymbol{\iota}, \boldsymbol{\lambda}, \boldsymbol{\eta})$ be a supersingular point in $Sh^{ss}_{U^{p}}$ whose associated $p$-divisible group is given by $\XX$ which we use as a base point to define the Rapoport-Zink space in Section \ref{RZ-def}. Then $I$ is defined to be the algebraic group over $\QQ$ whose $\QQ$-points $I(\QQ)$ are given by the group of quasi-isogenies in $\End(\bf{A})_{\QQ}$ which preserve the polarization $\boldsymbol{\lambda}$.

\subsection{Main theorem on the supersingular locus} We apply the results obtained previously to describe the supersingular locus of the quaternionic unitary Shimura variety.

\begin{theorem}
The $\FF$-scheme $Sh^{ss}_{U^{p}}$ is pure of dimension $2$. For $U^{p}$ sufficiently small, the irreducible components are isomorphic to the surface $x_{3}^{p}x_{0}-x^{p}_{0}x_{3}+x^{p}_{2}x_{1}-x^{p}_{1}x_{2}=0$. If two irreducible components intersects non-trivially then the intersection is either isomorphic to a $\PP^{1}$ or is a superspecial  point.
\end{theorem}
\begin{proof}
Using the uniformization Theorem \ref{uniformization}, the result follows from the main theorem for the Rapoport-Zink space Theorem \ref{main-result}.
\end{proof}

\section{Rapoport-Zink space for $\GSp(4)$ with paramodular level}

\subsection{Rapoport-Zink space with paramodular level}
In this section we treat the case when $G$ is the split symplectic group $\GSp(4)$ but where the level structure is the parahoric corresponding to the node labeled $1$ in the affine Dynkin diagram. This level structure is known as the paramodular level and the corresponding Shimura variety is well studied. In particular the structure of the supersingular locus is known. See for example \cite[Theorem 4.7]{Yu-doc06} where the author used the Hecke correspondence to the describe the supersingular locus. The purpose of this section is to revisit this description of the supersingular locus from the Bruhat-Tits perspective. We will also use the results in this section in \cite{Wang19b}.

First we define the Rapoport-Zink space we will study. Let $N$ be an isoclinic isocrystal of slope $\frac{1}{2}$ and height $4$ equipped with an alternating form $(\cdot,\cdot)$. We fix a $p$-divisible group $\XX$ whose associated isocrystal agrees with $N$ and a polarization $\lambda_{\XX}:\XX\rightarrow \XX^{\vee}$ which corresponds to the form $(\cdot, \cdot)$. We assume that the height of $\lambda_{\XX}$ is $2$. Recall that $\Nilp$ is the category of $W_{0}$-schemes on which $p$ is locally nilpotent. We consider the set valued functor $\mathcal{N}$ that sends $S\in\Nilp$ to the isomorphism classes of the triple $(X,  \lambda_{X}, \rho_{X})$ where
\begin{itemize} 
\item[-] $X$ is a $p$-divisible group of dimension $2$ and height $4$;
\item[-] $\lambda_{X}: X\rightarrow X^{\vee}$ is a polarization of height $2$;
\item[-] $\rho_{X}: X\times_{S} S_{0}\rightarrow \XX\times_{\FF}S_{0}$ is a quasi-isogeny on the closed subscheme of $S$ defined by the ideal given by $p$. 
\end{itemize}

For $\rho_{X}: X\times_{S} S_{0}\rightarrow \XX\times_{\FF}S_{0}$, we require that  $\rho_{X}^{-1}\circ\lambda_{\XX}\circ \rho_{X}=c(\rho)\lambda_{X}$ for a $\QQ_{p}$-multiplier $c(\rho)$.  The moduli problem $\calN$ is representable.

\begin{lemma}
The functor $\calN$ is representable by a separated formal scheme locally formally of finite type over $\Spf{W_{0}}$. Moreover $\calN$ is flat and formally smooth away from the superspecial points.
\end{lemma}

\begin{proof}
The representability of $\calN$ follows from \cite[Theorem 3.25]{RZ-Aoms} and second assertion on flatness and the singularity is proved in \cite{Yu-Proc11}.
\end{proof}

By passing to Dieudonn\'{e} modules, the $\FF$-points of the Rapoport-Zink space $\mathcal{N}(\FF)$ can be identified with the following set: 
\begin{equation}\label{set1split}
\begin{split}
\mathcal{N}(\FF) =\{M\subset N: pM\subset^{2} VM\subset^{2} M, p^{i+1}M^{\perp}\subset^{2} M\subset^{2} p^{i}M^{\perp} \text{ for some $i \in \ZZ$} \}.\\ 
\end{split}
\end{equation} 
The first condition comes from the fact that $X$ is of dimension $2$. And the second condition corresponds to the fact that $\lambda_{X}$ is a quasi-polarization of height $2$.

The formal scheme $\calN$ admits the following decomposition $\calN=\bigsqcup_{i\in\ZZ} \calN(i)$ where the the open and closed formal subscheme $\calN(i)$ of $\calN$ is defined by the condition that the quasi-isogeny $\rho_{X}$ has height $i$. Moreover, $\calN(i)$ is isomorphic to $\calN(0)$ and hence it suffices to only consider $\calN(0)$. Therefore we set $\calM=\calN_{red}(0)$ to be the underlying reduced scheme of $\calN(0)$. Then by the previous description, we arrive at the following lemma.

\begin{lemma}
The $\FF$-points of $\calM$ can be described  by
$$\mathcal{M}(\FF) =\{M\subset N; pM^{\perp}\subset^{2} M\subset^{2} M^{\perp} , pM\subset^{2} VM\subset^{2} M \}.$$
\end{lemma}

\subsection{The group $J_{b}$} This Rapoport-Zink space is associated to $G=\GSp(4)$ with level structure given by the paramodular parahoric. We define the $\sigma^{2}$-linear operator $\tau=V^{-1}F$. Then the slope $0$ isocrystal $(N,\tau)$ admits a factorization $N=C\otimes_{\QQ_{p^{2}}} K_{0}$ for a $\QQ_{p^{2}}$-vector space $C=N^{\tau=1}$. By letting $\Pi$ act on $C$ by $F$, we can consider the vector space $C$ as a $B_{p}$-module of rank $2$. The isocrytal $N$ gives an element $b\in B(G)$ where $B(G)$ is the set of $\sigma$-conjugate classes of $G$.  Recall the group $J_{b}$ is defined to be the $\sigma$-conjugate centralizer of $b$. This group can be computed as follows. 

\begin{lemma}
The group $J_{b}(\QQ_{p})$ can be identified with $\GU_{B_{p}}(C)(\QQ_{p})$.
\end{lemma}

\begin{proof}
The group $J_{b}(\QQ_{p})$ is by definition the group $$\{g\in \End(N)^{\times}_{\QQ}: gF=Fg, (gx, gy)=c(g)(x, y), c(g)\in \QQ\}.$$ Since $g$ commutes with $\tau=V^{-1}F$, it restricts to an action on $C$ and is determined by this restriction. Since $F$ acts on $C$ by $\Pi$, $$J_{b}(\QQ_{p})=\{g\in \End_{B_{p}}(C)^{\times}_{\QQ}: (gx, gy)=c(g)(x, y), c(g)\in \QQ\}=\GU_{B_{p}}(C)(\QQ_{p}).$$ \end{proof}

\section{Bruhat-Tits stratification for Rapoport-Zink space with paramodular level}
\subsection{Vertex lattices in $\GU_{B_{p}}(2)$} The group $\GU_{B_{p}}(2)$ splits over $\QQ_{p^{2}}$ and can be identified with $\GSp(4)$ over $\QQ_{p^{2}}$. We define vertex lattices $L$ in $C$ similarly as in Section \ref{vetex-lattices}. We use the same terminology as in Section \ref{vetex-lattices} and recall a vertex lattice of type $0$ corresponds to a lattice $L$ with $pL^{\perp}\subset^{4} L \subset^{0} L^{\perp}$, a vertex lattice of type $1$ corresponds to a lattice $L$ with $pL^{\perp}\subset^{2} L \subset^{2} L^{\perp}$ and a vertex lattice of type $2$ corresponds to a lattice $L$ with $pL^{\perp}\subset^{0} L \subset^{4} L^{\perp}$.  Here the dual $\perp$ is taken with respect to the alternating form $(\cdot,\cdot)$ restricted on $C$.

\begin{remark}
The relation between the group $\GSp(4)$ and $\GU_{B_{p}}(2)$ can be also demonstrated on the affine Dynkin diagram. For $\GSp(4)$, it is given by
\begin{displaymath}
 \xymatrix{\underset{0}\bullet \ar@2{->}[r] &\underset{1}\bullet&\underset{2}\bullet\ar@2{->}[l]}
\end{displaymath}
with Frobenius acting trivially on the diagram. For $\GU_{B_{p}}(2)$, it is given by the same affine Dynkin diagram
\begin{displaymath}
 \xymatrix{\underset{0}\bullet \ar@/^1pc/[rr]\ar@2{->}[r] &\underset{1}\bullet&\underset{2}\bullet\ar@2{->}[l] \ar@/_1pc/[ll]}
\end{displaymath}
with Frobenius acting on the diagram by permuting the nodes ${0}$ and $2$. Therefore the $0$-vertex and $2$-vertex should be identified when we consider vertex lattices for $\GU_{B_{p}}(2)$. We define a vertex lattice of type $\{0,2\}$ as a pair of lattices $(L_{0}, L_{2})$ where $L_{0}$ is a vertex lattice of type $0$ and $L_{2}$ is a vertex lattice of type $2$.

We also remark that $\GU_{B_{p}}(2)$ has two types of maximal parahoric subgroups by the above discussion on the affine Dynkin diagram.  A vertex lattice of type $1$ gives one of the maxiamal parahoric subgroups, and under the identification of $\GSp(4)(\QQ_{p^{2}})$ with $\GU_{B_{p}}(2)(\QQ_{p^{2}})$, this is the paramodular parahoric of  $\GSp(4)(\QQ_{p^{2}})$. On the other hand the parahoric subgroup of $\GU_{B_{p}}(2)$ corresponding to the a vertext lattice of type $\{0,2\}$ corresponds to the Siegel parahoric under $\GSp(4)(\QQ_{p^{2}})\cong\GU_{B_{p}}(2)(\QQ_{p^{2}})$.
\end{remark}

\begin{definition}\label{stratum-set-split} 
\hfill
\begin{enumerate}
\item For a vertex lattice $(L_{0}, L_{2})$ of type $\{0, 2\}$, define $\mathcal{M}_{L_{0},L_{2}}(\FF)$ to be the set $$ \{M\in \mathcal{M}(\FF): L_{2, W_{0}}\subset M\subset L_{0,W_{0}}\}.$$
\item For a vertex lattice $L_{1}$ of type $1$, define $\mathcal{M}_{L_{1}}(\FF)$ to be the set $$ \{M\in \mathcal{M}(\FF): L_{1,W_{0}}\subset M\}.$$
\end{enumerate}

We refer to these subsets as the lattice strata.
\end{definition}

To see that the lattice strata form a covering of the set $\calM(\FF)$, we need the following result.
\begin{proposition}\label{crucial_lemma_split}
Given $M\in\calM(\FF)$, there are $\tau$-stable lattices $L^{-}(M)$ and $L^{+}(M)$ of $N$ such that we have
 \begin{equation}\label{L+}
 M\subset L^{+}(M)\subset L^{+}(M)^{\perp}\subset M^{\perp}
 \end{equation}
 and 
 \begin{equation}\label{L-}
 pM^{\perp}\subset pL^{-}(M)^{\perp}\subset L^{-}(M)\subset M.
 \end{equation}
 \end{proposition}

Before we prove this Proposition, we need to recall a few basic facts about Dieudonn\'{e} modules that are stated in \cite[(11a)-(11d)]{NO-Ann80}. Define the $a$-number of the Dieudonn\'{e} module $M$ by  $$a(M)=\dim_{\FF} M/ (FM+VM).$$ Notice that $a(M)=2$ if and only if $M$ is superspecial. 
\begin{lemma}\label{NO_lemma}
Suppose $M$ has $a(M)=1$, then we have
\begin{enumerate}
\item  $F^{2}(M)/pM$ is the unique sub-$\FF[F,V]$-module of $F(M)/pM$ of rank $1$;
\item $F^{2}(M)+pM=FM\cap VM= V^{2}(M)+pM $.
\end{enumerate}
\end{lemma}
\begin{proof}
This is precisely the statement of $11(b), 11(c)$ and $11(d)$ in \cite{NO-Ann80}. 

\end{proof}

\begin{myproof}{Proposition}{\ref{crucial_lemma_split}}
Given $M\in\calM(\FF)$, then we set $L^{+}(M)=M+\tau(M)$ and $L^{-}(M)=M\cap \tau(M)$. First suppose $M$ is $\tau$-stable, then there is nothing to prove as $L^{+}(M)=L^{-}(M)=M$. Note in this case $M$ is $\tau$-stable and hence $M$ is superspecial with $a(M)=2$.

Now suppose $M$ is not $\tau$-stable. Since $M$ corresponds to a supersingular $p$-divisible group, $a(M)$ is at least $1$ and is equal to $1$ in this case. By Lemma \ref{NO_lemma}, we have $FM\cap VM= V^{2}M+pM$. Therefore $L^{-}(M)=M\cap \tau(M)=FM+V(M)$ and $L^{+}(M)=M+\tau(M)= V^{-1}(M\cap\tau(M))$. 
Since $FL^{-}(M)=F^{2}M+pM$ and $VL^{-}(M)=V^{2}M+pM$, $FL^{-}(M)=VL^{-}(M)$ and $L^{-}(M)$ is superspecial and $\tau$-stable. An easy computation shows that $M$ is $\tau$-stable if and only if $M^{\perp}$ is $\tau$-stable. Therefore $a(M^{\perp})=1$ and we can apply Lemma \ref{NO_lemma} to $M^{\perp}$ to get $F^{2}(M^{\perp})+pM^{\perp}$ is the unique colength $1$ Dieudonn\'{e} submodule of $F(M^{\perp})$. But $M\cap F(M^{\perp})$ is also colength $1$ in $F(M^{\perp})$ as otherwise $M=F(M^{\perp})$ will imply $a(M^{\perp})=2$. Indeed, take the dual of this equation give $M=V(M^{\perp})$. This implies that $M\cap F(M^{\perp})= F^{2}M^{\perp}+pM^{\perp}$ and therefore $F^{2}(M^{\perp})\subset M$. This is the same as $p\tau(M^{\perp})\subset M$ and hence $pM^{\perp}+p\tau(M^{\perp})\subset M\cap \tau(M)$.  But $pL^{-}(M)^{\perp}= pM^{\perp}+ p\tau (M^{\perp})$, thus $pL^{-}(M)^{\perp}\subset L^{-}(M)$. This finishes the proof for \eqref{L-}.

Since $L^{+}(M)=M+\tau(M)=V^{-1}(L^{-}(M))$,  $FL^{+}(M)= FV^{-1}(L^{-}(M))=\tau^{-1}L^{-}(M)=L^{-}(M)$ and $VL^{+}(M)= VV^{-1}(L^{-}(M))=L^{-}(M)$. This implies that $a(L^{+}(M))=2$. By the same argument as above using Lemma \ref{NO_lemma}, $F^{2}(M)+pM= pM^{\perp}\cap F(M)$. Therefore $\tau(M)\subset M^{\perp}$ and $L^{+}(M)=M+\tau(M)\subset M^{\perp}\cap \tau(M)^{\perp}=L^{+}(M)^{\perp}$.  This shows \eqref{L+}.

\end{myproof}

\begin{remark}
It is clear by considering the indices in \eqref{L+} and \eqref{L-} that if $M$ is $\tau$-stable, then $L^{+}(M)=L^{-}(M)=M$ corresponds to a vertex lattice of type $1$. If $M$ is not $\tau$-stable, then $L^{+}(M)$ corresponds to a vertex lattice of type $0$ and $L^{-}(M)$ corresponds to a vertex lattice of type $2$. Hence $(L^{-}(M), L^{+}(M))$ gives rise to a vertex lattice of type $\{0,2\}$. 
\end{remark}

We define similarly as in the quaternionic case 
\begin{equation}
\begin{split}
&\calM_{\{0,2\}}(\FF)=\bigcup_{L_{0}, L_{2}}\calM_{L_{0},L_{2}}(\FF)\\
&\text{ where $(L_{0}, L_{2})$ runs through all the vertex lattices of type $\{0,2\}$}.\\
&\calM_{\{1\}}(\FF)=\bigcup_{L_{1}}\calM_{L_{1}}(\FF)\\
&\text{ where $L_{1}$ runs through all the vertex lattices of type $1$}.\\
\end{split}
\end{equation}

By the above remark we obtain the following decomposition of the set $\calM(\FF)$.

\begin{lemma}
We have $\calM(\FF)=\calM_{\{0,2\}}(\FF)\cup \calM_{\{1\}}(\FF)$ and $\calM_{\{1\}}(\FF)$ consists of all the superspecial points.  
\end{lemma}

\begin{proof}
The first statement is a translation of Proposition \ref{crucial_lemma_split} and the second statement is clear as $M$ is superspecial if and only if $M$ is $\tau$-stable.
\end{proof}

\subsection{Bruhat-Tits stratification} Our next goal is to relate the set theoretic lattice stratum  $\calM_{L_{0}, L_{2}}(\FF)$ to a Deligne-Lusztig variety which in this case is simply $\PP^{1}(\FF)$. Suppose  $\calM_{L_{0}, L_{2}}(\FF)$ is non-empty, we define a two dimensional $\FF$-vector space by $V_{L_{0,2}}=L_{0,W_{0}}/L_{2,W_{0}}$.

\begin{proposition}\label{set-bruhat-tits}
Let $L_{0}$ be a vertex lattice of type $0$ and $L_{2}$ be a vertex lattice of type $2$. Suppose $\calM_{L_{0}, L_{2}}(\FF)$ is non-empty. There is a bijection between $\calM_{L_{0},L_{2}}(\FF)$ and $\PP^{1}(V_{L_{0,2}})(\FF)$. Moreover the superspecial points on $\calM_{L_{0},L_{2}}(\FF)$ correspond to the $\FF_{p^{2}}$-points on $\PP^{1}(V_{L_{0,2}})(\FF)$. 
\end{proposition}

\begin{proof}
We define a map $f: \calM_{L_{0}, L_{2}}(\FF)\rightarrow \PP^{1}(V_{L_{0,2}})(\FF)$ by sending $M\in \calM_{L_{0}, L_{2}}(\FF)$ to $M/L_{2,W_{0}}\subset L_{0, W_{0}}/L_{2,W_{0}}$. Conversely let $m\subset V_{L_{0,2}}=L_{0,W_{0}}/L_{2,W_{0}}$ be a point in $\PP^{1}(V_{L_{0,2}})$ and denote by $M\subset L_{0,W_{0}}$ the preimage of $m$ under the natural reduction map. By the proof of Proposition \ref{crucial_lemma_split}, we have $L_{2,W_{0}}=VL_{0,W_{0}}$. Then we have $VM\subset VL_{0,W_{0}}=L_{2,W_{0}}\subset M $ and $pM\subset pL_{0,W_{0}}=V^{2}L_{0,W_{0}}=VL_{2, W_{0}}\subset VM$. Hence $M$ is a Dieudonn\'{e} module. It is clear these two maps are inverse to each other and we have a bijection. It is also clear from the construction that the superspecial points are the $\FF_{p^{2}}$-points. 
\end{proof}

\begin{lemma}\label{L0_L2_intersection_split}
Let $L_{1}$ be a vertex lattice of type $1$. Then we can find a vertex lattice $L_{0}$ of type $0$ containing $L_{1}$ and a vertex lattice $L_{2}$ of type $2$ contained in $L_{1}$. 
\end{lemma}
\begin{proof}
This is proved the same way as Lemma \ref{L0_L2_intersection}  and we state it again for the reader's convenience. 
\end{proof}

The lemma and its proof shows that any superspecial point is contained in a lattice stratum of the form $\calM_{(L_{0}, L_{2})}$ for a vertex lattice $(L_{0}, L_{2})$ of type $\{0,2\}$. Let $(L_{0}, L_{2})$ and $(L^{\prime}_{0}, L^{\prime}_{2})$ be two vertex lattices of type $\{0, 2\}$. We study the set theoretic intersection between $\calM_{L_{0}, L_{2}}(\FF)$ and $\calM_{L^{\prime}_{0}, L^{\prime}_{2}}(\FF)$.

\begin{lemma}\label{L1inL02}
Suppose that the intersection $\calM_{L_{0}, L_{2}}(\FF)\cap \calM_{L^{\prime}_{0}, L^{\prime}_{2}}(\FF)$ is non-empty. Then $\calM_{L_{0}, L_{2}}(\FF)\cap \calM_{L^{\prime}_{0}, L^{\prime}_{2}}(\FF)$ consists of a superspecial point.
\end{lemma}

\begin{proof}
Given $M\in \calM_{L_{0}, L_{2}}(\FF)\cap \calM_{L^{\prime}_{0}, L^{\prime}_{2}}(\FF)$, we have $L_{0,W_{0}}\subset M\subset L_{2,W_{0}}$ and $L^{\prime}_{0,W_{0}}\subset M\subset L^{\prime}_{2,W_{0}}$. By considering the index, $L_{0,W_{0}}+ L^{\prime}_{0,W_{0}}=M$ and $L^{\prime}_{2,W_{0}}\cap L^{\prime}_{2,W_{0}}=M$. This forces $M$ to be $\tau$-stable and hence superspecial.

\end{proof}

\subsection{The isogeny trick} We are now ready to translate the set theoretic results in the previous sections to the scheme theoretic setting.  We will again rely on the isogeny trick used in Section \ref{Isogeny-trick}. Let $(L_{0}, L_{2})$ be a vertex lattice of type $\{0,2\}$ such that $\calM_{L_{0},L_{2}}(\FF)$ is non-empty. 

\begin{lemma}
The vertex lattices $L_{0}$ and $L_{2}$ give rise to $p$-divisible groups $\XX_{L_{0}}$ and $\XX_{L_{2}}$. Moreover there is a commutative diagram:
$$
\begin{tikzcd}
 \XX_{L_{2}} \arrow[r] \arrow[d,  "\rho_{L_{2}}"]
    & \XX^{\vee}_{L_{0}} \\
   \XX \arrow[r,  "\lambda_{\XX}"] & \XX^{\vee}  \arrow[u,  "\rho^{\vee}_{L_{0}}"] .
\end{tikzcd}$$ 
\end{lemma}

\begin{proof}
Since $\calM_{L_{0},L_{2}}(\FF)$ is non-empty, we have $VL_{0, W_{0}}=L_{2, W_{0}}$ and a chain 
\begin{equation}\label{chainL0_L2}
pM^{\perp}\subset pL^{\perp}_{2,W_{0}}\subset^{0} L_{2,W_{0}}\subset M\subset L_{0,W_{0}}\subset^{0} L^{\perp}_{0,W_{0}}\subset M^{\perp}\subset L^{\perp}_{2, W_{0}}.
\end{equation}
Therefore $$pL_{0,W_{0}}=pL^{\perp}_{0,W_{0}}\subset pL^{\perp}_{2,W_{0}}=L_{2,W_{0}}=VL_{0,W_{0}}\subset L_{0,W_{0}}$$ and $L_{0}$ is a Dieudonn\'{e} module. Since $L_{0}$ is $\tau$-stable, $V^{2}L_{0,W_{0}}=pL_{0,W_{0}}$. Then $pL_{2,W_{0}}\subset pL_{0,W_{0}}= V^{2}L_{0,W_{0}}=VL_{2,W_{0}}\subset VL_{0,W_{0}}=L_{2,W_{0}}$. The rest of the diagram follows from the chain \eqref{chainL0_L2}.

\end{proof}

Given $(X, \lambda_{X}, \rho_{X})\in \calM(R)$ for a $\FF$-algebra $R$, we define two quasi-isogenies by the following composites

\begin{equation}\label{rhoL0}
\rho_{X,L_{0}} : X \xrightarrow{\rho_{X}} \XX_{R} \xrightarrow{\rho^{-1}_{L_{0}}} \XX_{L_{0},R},
\end{equation}
and
\begin{equation}\label{rhoL2}
\rho_{X,L_{2}} :   \XX_{L_{2}, R}\xrightarrow{\rho_{L_{2}}}  \XX_{R} \xrightarrow{\rho^{-1}_{X}} X.
\end{equation}

We define the subfunctor $\calM_{L_{0},L_{2}}$ of $\calM$ over $\FF$ by classifying those $(X, \lambda_{X}, \rho_{X})$ such that $\rho_{X,L_{0}}$ and $\rho_{X,L_{2}}$ are both actual isogenies. 

\begin{proposition}
The functor $\calM_{L_{0}, L_{2}}$ is representable by a closed subscheme of $\calM$ and  $\calM_{L_{0}, L_{2}}$ is projective.
\end{proposition}

\begin{proof}
First of all, $\calM_{L_{0},L_{2}}$ is a closed subscheme by \cite[Proposition 2.9]{RZ-Aoms}. Moreover $\calM_{L_{0},L_{2}}$ is bounded in the sense of \cite[2.30]{RZ-Aoms}. This follows from the fact that we have isogenies
$$\XX_{L_{2}, R} \xrightarrow{\rho_{X,L_{2}}} X  \xrightarrow{\rho_{X,L_{0}}}\XX_{L_{0}, R}.$$
Hence $\calM_{L_{0},L_{2}}$ is a closed subscheme of a projective scheme which means it is projective itself by using \cite[Corollary 2.29]{RZ-Aoms}. \end{proof}

We define a map $f: \calM_{L_{0},L_{2}}\rightarrow \PP^{1}(V_{L_{0,2}})$ by sending $(X,\rho_{X}, \lambda_{X})$ to $$M/ L_{2, R}\subset L_{0,R}/L_{2, R}$$ for any $\FF$-algebra $R$ and where $M=\DD(X)(R)$ is the covariant Dieudonn\`{e} crystal of $X$ evaluated at $R$. 

\begin{proposition}
The map $f: \calM_{L_{0},L_{2}}\rightarrow \PP^{1}(V_{L_{0,2}})$ is an isomorphism.
\end{proposition}
\begin{proof}
Using the theory of windows of display \cite{Zink-pro99} in place of Dieudonn\'{e} modules, one can show $f$ is bijective on $k$-points for any field extension $k$ over $\FF$ following the same proof as in Proposition \ref{set-bruhat-tits}.  Since $\PP^{1}(V_{L_{0}, L_{2}})$ is normal and $\calM_{L_{0},L_{2}}$ is a projective scheme, we can apply Zariski's main theorem to show $f$ is an isomorphism.
\end{proof}

We set 
\begin{equation}\label{M0L0L2}\calM^{\circ}_{L_{0}, L_{2}}=\calM_{L_{0},L_{2}}- \calM_{\{1\}}\end{equation} and \begin{equation}\label{ML0L21}\calM_{L_{0},L_{2},\{1\}}=\calM_{L_{0},L_{2}}\cap \calM_{\{1\}}.\end{equation} 
The scheme $\calM_{L_{0}, L_{2}}$ admits the following decomposition called the Ekedahl-Oort stratification
\begin{equation}
\calM_{L_{0}, L_{2}}=\calM^{\circ}_{L_{0}, L_{2}}\sqcup \calM_{L_{0},L_{2},\{1\}}.
\end{equation}

\begin{corollary}\label{bruhat-tits-strat}
The set of superspecial points on $\calM_{L_{0}, L_{2}}$ is precisely the stratum $\calM_{L_{0},L_{2},\{1\}}$. The isomorphism $f$ respects the Ekedahl-Oort stratification in the sense that $\calM_{L_{0},L_{2},\{1\}}$ corresponds to the $\FF_{p^{2}}$-points on $\PP^{1}$ and $\calM^{\circ}_{L_{0}, L_{2}}$ corresponds to the complement of $\FF_{p^{2}}$-points on $\PP^{1}$. 
\end{corollary}
\begin{proof}
The statements can be checked on $\FF$-points and they follow from Proposition \ref{set-bruhat-tits}.
\end{proof}

\subsection{The main result in the paramodular Rapoport-Zink space} We summarize the results obtained from previous sections. We introduce the following Bruhat-Tits strata for $\calM$.  Define $\calM^{\circ}_{\{0,2\}}=\bigcup_{(L_{0}, L_{2})}\calM^{\circ}_{L_{0}, L_{2}}$ where $(L_{0}, L_{2})$ runs through all the vertices of type $\{0,2\}$ and with $\calM^{\circ}_{L_{0}, L_{2}}$ defined in \eqref{M0L0L2}. 

\begin{theorem}\label{main-result-param}
The formal scheme $\calN$ can be written as $\calN=\bigsqcup_{i\in \ZZ} \calN(i)$. The connected components $\calN(i)$ are all isomorphic to $\calN(0)$. We set $\calM=\calN_{red}(0)$ where $\calN_{red}(0)$ is the underlying reduced scheme of $\calN(0)$. 
\begin{enumerate}
\item Then $\calM$ is pure of dimension $2$ and $\calM$ can be decomposed into $$\calM=\calM^{\circ}_{\{0,2\}}\sqcup \calM_{\{1\}}.$$ We call this the Bruhat-Tits stratification of $\calM$.

\item The scheme $\calM^{\circ}_{\{0,2\}}=\bigcup_{(L_{0}, L_{2})}\calM^{\circ}_{L_{0}, L_{2}}$ where $(L_{0}, L_{2})$ runs through all vertex lattices of type $\{0,2\}$. The closure $\calM_{L_{0}, L_{2}}$ of $\calM^{\circ}_{L_{0}, L_{2}}$ is isomorphic to $\PP^{1}$. Moreover it admits a stratification $$\calM_{L_{0}, L_{2}}=\calM^{\circ}_{L_{0}, L_{2}}\sqcup \calM_{L_{0}, L_{2}, \{1\}}$$ called the Ekedhal-Oort stratification for $\calM_{L_{0}, L_{2}}$. The complement $\calM_{L_{0}, L_{2}, \{1\}}$ of $\calM^{\circ}_{L_{0}, L_{2}}$ in $\calM_{L_{0}, L_{2}}$ corresponds to the $\FF_{p^{2}}$-rational points of $\PP^{1}$.

\item The intersection between $\calM_{L_{0}, L_{2}}$ and $\calM_{L^{\prime}_{0}, L^{\prime}_{2}}$ for a vertex lattice $(L_{0}, L_{2})$ of type $\{0,2\}$ and another vertex $(L^{\prime}_{0}, L^{\prime}_{2})$ lattice of type $\{0,2\}$ if nonempty is a superspecial point. 
\end{enumerate}
\end{theorem}
\begin{proof}
The first statement $(1)$ is seen from Lemma \ref{L0_L2_intersection}. The second statement $(2)$ is proved in Proposition \ref{bruhat-tits-strat}. The third statement $(3)$ is proved in Lemma \ref{L1inL02}.
\end{proof}

\subsection{Application to the supersingular locus}  Let $V$ be a vector space of dimension $4$ over $\QQ$. We assume that $V$ is equipped with a symplectic form $(\cdot,\cdot)$. Then we define $$G(\QQ)=\{g\in \GL(V); (gx, gy)=c(g)(x, y), c(g)\in \QQ^{\times}\}.$$ That is $G=\GSp(4)$. Let $h: \GG_{m}\rightarrow G(\CC)$ be the cocharacter sending $z$ to $\text{diag}(z,z,1,1)$. Moreover $h$ defines a decomposition $V_{\CC}=V_{1}\oplus V_{2}$ where $h(z)$ acts on $V_{1}$ by $z$ and on $V_{2}$ by $\bar{z}$. We fix an open compact subgroup $U=U_{p}U^{p}$ of $G(\AAA_{f})$ and we assume $U^{p}$ is sufficiently small. We also assume that there is a lattice $\Lambda$ in $V$ such that $\Lambda$ is paramodular in the sense that $p\Lambda^{\perp}\subset^{2}\Lambda\subset^{2}\Lambda^{\perp}$. We choose $U_{p}$ that stabilize $\Lambda\otimes \ZZ_{p}$.

To $(V, (\cdot,\cdot), h ,\Lambda)$, we consider the following moduli problem $\gothS h_{U^{p}}$ over $\ZZ_{(p)}$: for a scheme $S$ over $\ZZ_{(p)}$  we associate the set of isomorphism classes of the triple $(A, \lambda, \eta )$ where:
\begin{itemize}
\item[-]  $A$ is an abelian scheme of relative dimension $2$ over $S$;
\item[-]  $\lambda: A\rightarrow A^{\vee}$ is a polarization of degree $p^{2}$;
\item[-] $\eta: V\otimes_{\QQ} {\AAA}^{(p)}_{f} \rightarrow V^{(p)}(A)$ is a $U^{p}$-orbit of  isomorphisms. Here $$V^{p}(A)=\prod_{p^{\prime}\neq p}\Ta_{p^{\prime}}(A)\otimes \AAA_{f}^{(p)}$$ is the prime to $p$-part of the rational Tate module of $A$. \end{itemize}

This moduli problem is representable by a quasi-projective variety $\Sh_{U^{p}}$  over $\ZZ_{(p)}$ which is well-known.
We are interested in the supersingular locus $Sh^{ss}_{U^{p}}$ in $\Sh_{U^{p},\FF}$ which is considered as a closed reduced subscheme. The uniformization theorem of Rapoport-Zink \cite[Theorem 6.1]{RZ-Aoms} translates this problem to problem of describing the corresponding Rapoport-Zink space.

\begin{theorem}\label{uniformization_split}
There is an isomorphism of $\FF$-schemes 
$$Sh^{ss}_{U^{p}}\cong I(\QQ)\backslash \calN_{red}\times  G(\AAA^{p}_{f})/U^{p}.$$
\end{theorem} 

Here $I$ is an inner form of $G$ such that $I(\QQ_{p})=J_{b}(\QQ_{p})$ defined exactly as we did in the quaternionic unitary case in Theorem \ref{uniformization}. 

\subsection{Main theorem on the supersingular locus} We apply the results obtained previously to describe the supersingular locus of the Siegel threefold with paramodular level structure.
\begin{theorem}
The $\FF$-scheme $Sh^{ss}_{U^{p}}$ is pure of dimension $1$. For $U^{p}$ sufficiently small, the irreducible components are isomorphic to the projective line $\PP^{1}$. The intersection of two irreducible components is a superspecial  point.
\end{theorem}
\begin{proof}
Using the uniformization Theorem \ref{uniformization_split}, the result follows from the main theorem for the Rapoport-Zink space Theorem \ref{main-result-param}.
\end{proof}

\section{Comparison with Affine Deligne-Lusztig varieties}

In this final section we would like to point out how our results fit in the results proved in G\"{o}rtz and He \cite{GH-Cam15} in terms of the affine Deligne-Lusztig varieties. In the following we will abbreviate affine Deligne-Lusztig variety as ADLV.
\subsection{Affine Deligne Lusztig variety} Let $F$ be a finite extension of $\QQ_{p}$ and $\breve{F}$ be the completion of the maximal unramified extension of $F$. Let $G$ be a connected reductive group over $F$ and we write $\breve{G}$ its base change to $\breve{F}$. Then $\breve{G}$ is quasi-split and we choose a maximal split torus $S$ and denote by $T$ its centralizer.  We know $T$ is a maximal torus and we denote by $N$ its normalizer. The relative Weyl group is defined to be $W_{0}=N(\breve{F})/ T(\breve{F})$. Let $\Gamma$ be the Galois group of $\breve{F}$ and we have the following Kottwitz homomorphism \cite{Ko97}:
$$\kappa_{G}: G(\breve{F})\rightarrow X_{*}(\breve{G})_{\Gamma}.$$
Denote by $\widetilde{W}$ the Iwahori Weyl group of $\breve{G}$ which is by definition given by $\widetilde{W}=N(\breve{F})/ T(\breve{F})_{1}$ where $T(\breve{F})_{1}$ is the kernel of the Kottwitz homomorphism for $T(\breve{F})$. Let $\breve{\mathfrak{B}}(G)$ be the Bruhat-Tits building of $G$ over $\breve{F}$. The choice of $S$ determines an standard apartment $\breve{\mathfrak{A}}$ which $\widetilde{W}$ acts on by affine transformations. We fix a $\sigma$-invariant alcove $\mathfrak{a}$ and a special vertex of $\mathfrak{a}$. Inside the Iwahori Weyl group $\widetilde{W}$, there is a copy of the affine Weyl group $W_{a}$ which can be identified with $N(\breve{F})\cap G(\breve{F})_{1}/ T(\breve{F})_{1}$ where $G(\breve{F})_{1}$ is the kernel of the Kottwitz morphism for $G$. The group $\widetilde{W}$ is not quite a Coxeter group while $W_{a}$ is generated by the set of simple affine reflections denoted by $\tilde{\mathbb{S}}$ and $(\widetilde{W}, \tilde{\mathbb{S}})$ form a Coxeter system. We in fact have $\widetilde{W}= W_{a}\rtimes \Omega$ where $\Omega$ is the normalizer of a fixed base alcove and more canonically $\Omega=X_{*}(T)_{\Gamma}/ X_{*}(T_{sc})_{\Gamma}$ where $T_{sc}$ is the preimage of $T\cap G_{der}$ in the simply connected cover $G_{sc}$ of $G_{der}$.

Let $\mu\in X_{*}(T)$ be a minuscule cocharacter of $G$ over $\breve{F}$ and $\lambda$ its image in $X_{*}(T)_{\Gamma}$.  We denote by $\tau$ the projection of $\lambda$ in $\Omega$. The \emph{admissible subset} of $\widetilde{W}$ is defined to be
$$\Adm(\mu)=\{w\in \widetilde{W}: w \leq x(\lambda) \text{ for some }x\in W_{0} \}.$$
Here $\lambda$ is considered as a translation element in $\widetilde{W}$. Let $K\subset\tilde{\mathbb{S}}$ and $\breve{K}$ its corresponding parahoric subgroup. Let $\widetilde{W}_{K}$ be the subgroup defined by $(N(\breve{F})\cap \breve{K})/T(\breve{F})_{1}$. We have the following decomposition $\breve{K}\backslash G(\breve{F})/\breve{K}= \widetilde{W}_{K}\backslash \widetilde{W}/ \widetilde{W}_{K}$. Therefore we can define a relative position map by 
$$\inv: G(\breve{F})/\breve{K}\times G(\breve{F})/\breve{K}\rightarrow \widetilde{W}_{K}\backslash \widetilde{W}/ \widetilde{W}_{K}.$$

For $w\in \widetilde{W}_{K}\backslash \widetilde{W}/ \widetilde{W}_{K}$ and $b\in G(\breve{F})$, we define the \emph{affine Deligne-Lusztig variety}
to be the set 
$$X_{w}(b)=\{g\in G(\breve{F})/\breve{K}: \inv(g, b\sigma(g))=w\}.$$
Thanks to the work of \cite{BS-Inv17} and \cite{Zhu-Ann17}, this set can be viewed as an ind-closed-subscheme in the affine flag variety $\breve{G}/\breve{K}$. In this article, we only consider it as a set. The Rapoport-Zink space is not directly related to the affine Deligne-Lusztig variety but rather to the following union of affine Deligne-Lusztig varieties
$$X(\mu, b)_{K}=\{g\in G(\breve{F})/ \breve{K}: g^{-1}b\sigma(g)\in \breve{K}w\breve{K}, w\in \Adm(\mu) \}.$$
We recall that the group $J_{b}$ is defined as the $\sigma$-centralizer of $b$ that is $$J_{b}(R)=\{g\in G(R\otimes_{F}\breve{F}): g^{-1}b\sigma(g)=b\}$$ for any $F$-algebra $R$. In the following we will assume that $b$ is basic and in this case $J_{b}$ is an inner form of $G$.

\subsection{Coxeter type ADLV}We define $\Adm^{K}(\mu)$ to be the image of  $\Adm(\mu)$ in $\widetilde{W}_{K}\backslash\widetilde{W}/\widetilde{W}_{K}$ and $^{K}\widetilde{W}$ to be the set of elements of minimal length in $\widetilde{W}_{K}\backslash\widetilde{W}$.  We define the set $\mathrm{EO}^{K}(\mu)=\Adm^{K}(\mu)\cap ^{K}\widetilde{W}$. For $w\in W_{a}$, we set 
$$\mathrm{supp}_{\sigma}(w\tau)=\bigcup_{n\in \ZZ}(\tau\sigma)^{n}(\mathrm{supp}(w)).$$
If the length $l(w)$ of $w$ agrees with the cardinality of the set $\mathrm{supp}_{\sigma}(w\tau)/\langle\tau\sigma\rangle$, we say $w\tau$ is a \emph{$\sigma$-Coxeter element}. We denote by $\mathrm{EO}^{K}_{\sigma,\mathrm{cox}}(\mu)$ the subset of $\mathrm{EO}^{K}(\mu)$ such that $w$ is a $\sigma$-Coxeter element and $\mathrm{supp}_{\sigma}(w)$ is not equal to $\widetilde{\mathbb{S}}$. A \emph{$K$-stable piece} is a subset of $G(\breve{F})$ of the form $\breve{K}\cdot_{\sigma}IwI$ where $\cdot_{\sigma}$ means $\sigma$-conjugation, $I$ is an Iwahori subgroup and $w\in {^{K}\widetilde{W}}$. Then we define the Ekedahl-Oort stratum attached to $w\in \mathrm{EO}^{K}(\mu)$ of $X(\mu, b)_{K}$ by the set $$X_{K,w}(b)=\{g\in G(\breve{F})/\breve{K}: g^{-1}b\sigma(g)\in  \breve{K}\cdot_{\sigma}IwI\}.$$  Then by \cite[1.4]{GHN16} we have the following EO-stratification 
\begin{equation}X(\mu, b)_{K}=\bigcup_{w\in\mathrm{EO}^{K}(\mu)}X_{K,w}(b).\end{equation} 
The case when 
\begin{equation}\label{ADLV-EO}X(\mu, b)_{K}=\bigcup_{w\in\mathrm{EO}^{K}_{\sigma,\mathrm{cox}}(\mu)}X_{K,w}(b)\end{equation}
is particular interesting and when this happens we say the datum $(G, \mu, K)$ is of Coxeter type. The datum $(G, \mu, K)$ being Coxeter type or not depends only on the associated datum $(\widetilde{W}, \lambda, K, \sigma)$ where $\widetilde{W}$ is the Iwahori Weyl group of $G$, $\lambda$ is the image of $\mu$ in $X_{*}(T)_{\Gamma}$ and $\sigma$ is the induced automorphism on the affine Dynkin diagram by the Frobenius $\sigma$ on $\breve{G}$. The set of $(G, \mu, K)$ that is of Coxeter type is classified in \cite[Theorem 5.12]{GH-Cam15}. This includes the two cases we studied in the previous sections.
\begin{itemize}
\item[-] The quaternionic unitary case  corresponds to $G=\GU_{B_{p}}(2)$ and the datum $$(\widetilde{C}_{2}, \omega^{\vee}_{2}, \widetilde{\mathbb{S}}-\{1\}, \tau_{2})$$ where $\sigma$ acts on the affine Dynkin diagram by the image $\tau_{2}$ of $\omega^{\vee}_{2}$ in $\Omega$.

\item[-] The paramodular case corresponds to $G=\GSp(4)$ and the datum $$(\widetilde{C}_{2}, \omega^{\vee}_{2}, \widetilde{\mathbb{S}}-\{1\}, id)$$
where $\sigma$ acts on the affine Dynkin diagram trivially.

\end{itemize}

\subsection{Bruhat-Tits stratification of ADLV} Now we assume that $K$ is a maximal proper subset of $\tilde{\mathbb{S}}$ such that $\sigma(K)=K$. Consider the following set
$$\mathcal{J}=\{\Sigma\subset\tilde{\mathbb{S}}: \emptyset\neq \Sigma\text{ is }  \tau\sigma\text{-stable}\text{ and }d(v)=d(v^{\prime})\text{ for every } v,v^{\prime}\in \Sigma\}$$ where $d(v)$ is the distance between $v$ and the unique vertex not in $K$.
In fact every $w\in \mathrm{EO}^{K}_{\sigma,\mathrm{cox}}(\mu)$ corresponds to a $\Sigma\in \mathcal{J}$ and we write $w$ as $w_{\Sigma}$.  If $(G,\mu, K)$ is of Coxeter type, for any $w_{\Sigma}\in\mathrm{EO}^{K}_{\sigma,\mathrm{cox}}(\mu)$, \begin{equation}\label{EO-DL}X_{K,w_{\Sigma}}(b)=\bigcup_{i\in J_{b}/J_{b}\cap \breve{K}_{\tilde{\mathbb{S}}-\Sigma}}i X(w_{\Sigma}).\end{equation} Here $\breve{K}_{\tilde{\mathbb{S}}-\Sigma}$ is the parahoric subgroup associated to the set $\tilde{\mathbb{S}}-\Sigma$ and $X(w_{\Sigma})$ is a classical Deligne-Lusztig variety defined by $$X(w_{\Sigma})=\{g\in \breve{K}_{\mathrm{supp}_{\sigma}(w_{\Sigma})}/\breve{I}: g^{-1}\tau\sigma(g)\in \breve{I}w\breve{I}\}.$$ This is a Deligne-Lusztig variety attached to the maximal reductive quotient $G_{w}$ of the special fiber of $\breve{K}_{\tilde{\mathbb{S}}-\Sigma}$. Combine \eqref{ADLV-EO} and \eqref{EO-DL} we arrive at the following \emph{Bruhat-Tits stratification} of $X(\mu, b)_{K}$:
\begin{equation}
X(\mu, b)_{K}=\bigcup_{J_{b}/J_{b}\cap \ker(\kappa_{G})}\bigcup_{w_{\Sigma}\in  \mathrm{EO}^{K}_{\sigma,\mathrm{cox}}} \mathcal{X}^{\circ}_{\Sigma}
\end{equation}
where 
\begin{equation}\label{MSigma}
\mathcal{X}^{\circ}_{\Sigma}=\bigcup_{i\in J_{b}\cap\ker(\kappa_{G})/J_{b}\cap \breve{K}_{\tilde{\mathbb{S}}-\Sigma}}i X(w_{\Sigma}).
\end{equation}
Here the index set is related to the Bruhat-Tits building of $J_{b}$ in the following way. The group  $J_{b}\cap\ker(\kappa_{G})$ acts on the set of faces of type $\Sigma$ transitively and  $J_{b}\cap \breve{K}_{\tilde{\mathbb{S}}-\Sigma}$ is precisely the stabilizer of the face of type $\Sigma$ in the base alcove.

\subsubsection{Quaternionic unitary case} In the quaternionic unitary case, we can compute
\begin{center}
\begin{tabular}{lllll}
$\Sigma$                                     & \{1\}                    & \{0,2\}                     & \{0\}                               & \{2\} \\
$w_{\Sigma}$                               & $\tau $     & $s_{1}\tau$ & $s_{1}s_{2}\tau$ &        $s_{1}s_{0}\tau$ \\
$\tilde{\mathbb{S}}-\Sigma$             & \{0,2\}                  & \{1\}                       & \{1,2\}                             & \{0,1\}\\
$\mathrm{supp}_{\sigma}(w_{\Sigma})$ & $\emptyset $& \{1\}                       & \{1,2\}                             & \{0,1\}.                            
\end{tabular}
\end{center}

In this case the Deligne-Lusztig varieties $X(w_{\{0\}})$ and $X(w_{\{2\}})$ agrees with $X_{B}(w_{2})$ in Theorem \ref{DL-stratum}. The Deligne-Lusztig variety $X(w_{\{0,2\}})$ agrees with $X_{B}(w_{1})$ and $X(w_{\{1\}})$ is $0$-dimensional and agrees with $X_{P_{\{2\}}}(1)$. Therefore we have the following comparison between the Bruhat-Tits stratification for ADLV and Bruhat-Tits stratification for the Rapoport-Zink space studied in the quaternionic unitary case. Recall that for $\calM$ the Bruhat-Tits stratification in Theorem \ref{main-result} reads \begin{equation}\label{BT-stra-AFDL1}\calM=\calM^{\circ }_{\{0\}}\sqcup\calM^{\circ}_{\{2\}}\sqcup \calM^{\circ}_{\{0,2\}}\sqcup \calM_{\{1\}}.\end{equation}
\begin{itemize}
\item[-] $\mathcal{X}^{\circ}_{\{0\}}$ in \eqref{MSigma} is identified with $\calM^{\circ }_{\{0\}}$ in \eqref{BT-stra-AFDL1};
\item[-] $\mathcal{X}^{\circ}_{\{2\}}$ in \eqref{MSigma} is identified with $\calM^{\circ }_{\{2\}}$ in \eqref{BT-stra-AFDL1};
\item[-] $\mathcal{X}^{\circ}_{\{0, 2\}}$ in \eqref{MSigma}  is identified with $\calM^{\circ }_{\{0,2\}}$ in \eqref{BT-stra-AFDL1};
\item[-] $\mathcal{X}^{\circ}_{\{1\}}$ in \eqref{MSigma}  is identified with $\calM_{\{1\}}$ in \eqref{BT-stra-AFDL1} and is the set of superspecial points.
\end{itemize}

\subsubsection{Paramodular Siegel case} In the paramodular Siegel case, we can compute
\begin{center}
\begin{tabular}{lllll}
$\Sigma$                                     & \{1\}                    & \{0,2\}   \\
$w_{\Sigma}  $                               & $\tau $     & $s_{1}\tau$   \\
$\tilde{\mathbb{S}}-\Sigma$             & \{0,2\}                  & \{1\}   \\
$\mathrm{supp}_{\sigma}(w_{\Sigma})$                             & $\emptyset $           & \{1\}.                                                
\end{tabular}
\end{center}

In this case the Deligne-Lusztig varieties $X(w_{\{1\}})$ is $0$-dimensional and $X(w_{\{0,2\}})$ is isomorphic the complement of the $\FF_{p^{2}}$-points in $\PP^{1}$. Therefore we have the following comparison between the Bruhat-Tits stratification for ADLV and Bruhat-Tits stratification for the Rapoport-Zink space studied in the paramodular Siegel case. Recall that for $\calM$ the Bruhat-Tits stratification in Theorem \ref{main-result-param} reads \begin{equation}\label{BT-stra-AFDL2}\calM=\calM^{\circ }_{\{0,2\}}\sqcup \calM_{\{1\}}.\end{equation}

\begin{itemize}
\item[-] $\mathcal{X}^{\circ}_{\{0,2\}}$ in \eqref{MSigma} can be identified with $\calM^{\circ }_{\{0,2\}}$ in \eqref{BT-stra-AFDL2};
\item[-] $\mathcal{X}^{\circ}_{\{1\}}$ in \eqref{MSigma}  can be identified with $\calM_{\{1\}}$ in \eqref{BT-stra-AFDL2} and is the set of superspecial points.
\end{itemize}

\end{document}